\documentclass[12pt]{amsart}

\usepackage{amsmath,amssymb,amsfonts,graphics,amsthm}
\usepackage{mathabx}
\usepackage{verbatim}
\usepackage{fullpage}
\usepackage{tikz}
\usepackage{graphicx}

\newcommand{\N}{\mathbb{N}}

\newcommand{\R}{\mathbb{R}}

\newcommand{\C}{\mathbb{C}}
\newcommand{\F}{\mathbb{F}}
\newcommand{\G}{\mathbb{G}}
\newcommand{\W}{\mathbb{W}}

\newcommand{\T}{\mathbb{T}}

\newcommand{\mc}{\mathcal}
\newcommand{\mf}{\mathfrak}

\newcommand{\ip}[2]{\langle #1,#2 \rangle}

\newcommand{\Tr}{\text{Tr}}

\newcommand{\Irr}{\text{Irr}}
\newcommand{\Pol}{\text{Pol}}

\newtheorem{thm}{Theorem}[section]
\newtheorem{cor}[thm]{Corollary}
\newtheorem{lem}[thm]{Lemma}
\newtheorem{prop}[thm]{Proposition}

\theoremstyle{definition}
\newtheorem{defn}[thm]{Definition}

\newtheorem{rem}[thm]{Remark}
\newtheorem{rems}[thm]{Remarks}
\newtheorem{notat}[thm]{Notation}

\newtheorem{assumption}[thm]{Assumption}
\numberwithin{equation}{section}

\begin{document}

\title[$L_p$-Representations]
{$L_p$-Representations of Discrete Quantum Groups} 

\date{\today}

\author{Michael Brannan$^1$ and Zhong-Jin Ruan$^2$}

\address{Michael Brannan and Zhong-Jin Ruan: Department of Mathematics, University  of Illinois at Urbana-Champaign, Urbana, IL 61801, USA}
\email{mbrannan@illinois.edu, ruan@illinois.edu}

\thanks{$^1$ The first author was partially supported by an NSERC Postdoctoral Fellowship. $^2$ The second author was partially supported by the Simons Foundation.}

\begin{abstract}
Given a locally compact quantum group $\G$, we define and study representations and C$^\ast$-completions of the convolution algebra $L_1(\G)$ associated with various linear subspaces of the multiplier algebra $C_b(\G)$.  For discrete quantum groups $\G$, we investigate the left regular representation, amenability and the Haagerup property in this framework.  When $\G$ is unimodular and discrete, we study in detail the C$^\ast$-completions of $L_1(\G)$ associated with the non-commutative $L_p$-spaces $L_p(\G)$.  As an application of this theory, we characterize (for each $p \in [1,\infty)$) the positive definite functions on unimodular orthogonal and unitary free quantum groups $\G$ that extend to states on the $L_p$-C$^\ast$-algebra of $\G$.  Using this result, we construct uncountably many new examples of exotic quantum group norms for compact quantum groups. 
\end{abstract}
\maketitle

\section{Introduction} \label{section:intro}

In the theory of operator algebras, the C$^\ast$-algebras associated with locally compact groups play a prominent role, providing many interesting examples and phenomena which motivate the general theory.  Two particularly important C$^\ast$-algebras associated with any locally compact group $G$ are the \textit{full} and \textit{reduced} C$^\ast$-algebras, $C^*(G)$ and $C^*_{\lambda}(G)$, respectively.  In general, there always exists a natural surjective $\ast$-homomorphism ${\hat \pi}_\lambda:C^*(G) \to C^*_\lambda(G)$.  Moreover, it is known that $G$ is amenable if and only if $C^*(G)$ and $C^*_\lambda(G)$ are isomorphic, which happens if and only if ${\hat \pi}_\lambda$ is injective.  When $G$ is non-amenable, there generally exist many intermediate quotient C$^\ast$-algebras 
\[C^*(G) \to A \to C^*_\lambda(G).\] 

Note that both the full and reduced C$^\ast$-algebras of a locally compact group $G$ carry additional structure coming from the underlying group $G$.  More precisely, they are 
Hopf C$^\ast$-algebras in the sense of \cite{VaVan} and the corresponding quotient map ${\hat \pi}_\lambda:C^*(G) \to C^*_\lambda(G)$ is a morphism in the category of Hopf C$^\ast$-algebras.  Therefore, it is natural to investigate the existence and structure of the intermediate (or \textit{exotic}) Hopf C$^\ast$-algebra quotients $C^*(G) \to A \to C^*_\lambda(G)$. 

In a recent paper \cite{BrGu12}, Brown and Guentner made a significant contribution in this direction for discrete groups $G$.  They showed that for any ideal $D \subset \ell_\infty(G)$, one can construct a corresponding Hopf C$^\ast$-algebra $C^*_D(G)$--the \textit{$D$-C$^\ast$-algebra of $G$}--by completing the group algebra $\C[G]$ in a suitable C$^\ast$-norm constructed from $D$.  When $D =\ell_\infty(G)$, the corresponding ideal completion of $\C[G]$ is the full C$^\ast$-algebra $C^*(G)$ and when $D = C_c(G)$ (the ideal of finitely supported functions) or when $D=\ell_p(G)$ for $1 \le p \le 2$, one recovers the reduced C$^\ast$-algebra
$C^*_{\lambda}(G)$.  For intermediate ideals $C_c(G) \subsetneq D \subsetneq \ell_\infty(G)$  (e.g. $D = C_0(G)$ or $\ell_p(G)$ for $2 < p < \infty$), the resulting ideal completions $C^*_D(G)$ turn out to be quite mysterious and interesting intermediate Hopf C$^\ast$-algebras.  In particular, Brown and Guentner showed that for a free group $\F_k$ on $2 \le k < \infty$ generators, there exists a $p \in (2, \infty)$ such that the  intermediate Hopf C$^\ast$-quotient $C^*(\F_k) \to C^*_{\ell_p}(\F_k) \to C^*_\lambda (\F_k)$ is exotic.

In a follow-up paper \cite{Ok12}, Okayasu gave a detailed study of the structure of  $C^*_{\ell_p}(\F_k)$.  The main result in \cite{Ok12} (which was also obtained  independently by Higson and Ozawa) is that for all $2 < p < \infty$, the Hopf C$^\ast$-algebras $C^*_{\ell_p}(\F_k)$ are mutually non-isomorphic (as Hopf C$^\ast$-algebras).  Okayasu obtains this result by first showing that $\F_k$ satisfies a suitable $\ell_q$ version of the property of rapid decay for $1 \le q \le 2$ with respect to the standard conditionally negative definite word length function on $\F_k$.  This $\ell_q$ property of rapid decay is then used to characterize, for each $2 \le p < \infty$, the positive definite functions on $\F_k$ that extend to states on $C^\ast_{\ell_p}(\F_k)$ \cite[Theorem 3.4]{Ok12}.  The non-isomorphism result follows from this characterization.  Another remarkable consequence of this characterization is that each  C$^\ast$-algebra $C^*_{\ell_p}(\F_k)$ admits a unique trace \cite[Corollary 3.9]{Ok12}.   

In this paper, our goal is to extend the theory of ideal completions to the general context of locally compact quantum groups.  More generally, given a locally compact quantum group $\G$ and a linear subspace $D$ of the multiplier C$^\ast$-algebra of ``bounded continuous functions'' $C_b(\G) = M(C_0(\G))$, we define and study \textit{$D$-representations} of $\G$.  Roughly speaking, these are representations $\pi:L_1(\G)\to \mc B(H)$ with the property that for each pair of vectors $\xi,\eta$ in a dense subspace of $H$,  the corresponding coefficient function $\varphi^\pi_{\xi,\eta}$ belongs to $D$.   If the subspace $D$ is sufficiently large, one can then define a C$^\ast$-norm $\|\omega\|_D = \sup_\pi\{\|\pi(\omega)\|\}$ on $L_1(\G)$, where the supremum runs over all $D$-representations $\pi$.  The closure $C^*_D(\G)$ of $L_1(\G)$ in this new norm is called the \textit{$D$-C$^\ast$-algebra of $\G$}.

With an eye towards studying concrete examples, we focus for a large part of this paper on discrete quantum groups.  In the discrete case, we generalize many of the known results for $D$-C$^\ast$-algebras of discrete groups. In particular we show that whenever $D$ is a subalgebra of $C_b(\G)$ containing $C_c(\G)$ (the algebra of ``finitely supported functions on $\G$'') then $C^*_D(\G)$ carries a coproduct $\hat \Delta_D$ making it a compact quantum group C$^\ast$-algebra (Proposition \ref{prop:cqg_structure}).  Using Pontryagin duality, we can interpret $C^*_D(\G)$ as a C$^\ast$-completion of the algebra $\Pol(\hat \G)$ of  polynomial functions on the dual compact quantum group $\hat \G$ and therefore Proposition \ref{prop:cqg_structure} provides a novel approach to the construction of exotic quantum group norms for compact quantum groups in the sense of Kyed and Soltan \cite{KySo}.   Pursuing this avenue (with inspiration from the work of Brown-Guentner and Okayasu), we consider unimodular discrete quantum groups and their corresponding non-commutative $L_p$-spaces $L_p(\G)$.  We show that in the quantum realm, $L_p$-C$^\ast$-algebras of unimodular discrete quantum groups yield new and interesting examples of exotic compact quantum group norms.  More precisely, we study in detail Van Daele and Wang's unimodular orthogonal free quantum groups (see \cite{VaWa}) and prove the following theorem. 

\begin{thm}[Theorems \ref{thm:free_orthogonal} and \ref{thm:uniquetrace} and Remark \ref{rem:exotic}] \label{thm:exotic}
Let $N\ge 3$, $F \in \text{GL}_N(\C)$ a multiple of a unitary matrix with $F\bar F \in \R 1$, and let $\F O_F$ be the corresponding unimodular orthogonal free quantum group with dual compact quantum group $O^+_F = \widehat {\F O_F}$.
\begin{enumerate}
\item  For each $2 <p < \infty$, the canonical sequence of quotient maps 
 \[C_u(O^+_F) \to C^*_{L_{p}(\F O_F)}(\F O_F) \to C(O^+_F) \]
from the universal quantum group C$^\ast$-algebra $C_u(O^+_F)$ to the reduced quantum group C$^\ast$-algebra $C(O^+_F)$ (extending the identity map on $L_1(\F O_F)$) are not injective.  In particular, each $L_p$-C$^\ast$-algebra norm defines an exotic quantum group norm on $\Pol(O^+_F)$ in the sense of Kyed and Soltan \cite{KySo}. 
\item For each $2 < p \ne p' < \infty$, the Hopf C$^\ast$-algebras $C^*_{L_{p}(\F O_F)}(\F O_F)$ and $C^*_{L_{p'}(\F O_F)}(\F O_F)$ are not isomorphic (as Hopf C$^\ast$-algebras).  Moreover, $C^*_{L_{p}(\F O_F)}(\F O_F)$ is not isomorphic (as an abstract C$^\ast$-algebra) to either $C_u(O^+_F)$ or $C(O^+_F)$.
\item For each $2 \le p < \infty$,  the C$^\ast$-algebra $C^*_{L_{p}(\F O_F)}(\F O_F)$ admits a unique trace.
\end{enumerate}
\end{thm}

The results of the above theorem precisely parallel what is known about the $\ell_p$-C$^\ast$-algebras of free groups from \cite{Ok12}.  Moreover, our general strategy for proving the above results is inspired by \cite{Ok12}: We first prove a non-commutative $L_q$-version of the property of rapid decay for the discrete quantum groups $\F O_F$ (Proposition \ref{prop:qHaagerup}) and use this geometric-analytic property in conjunction with the Haagerup property for $\F O_F$ \cite{BrAP, DeFrYa} to characterize the positive definite functions on $\F O_F$ that extend to states on the corresponding $L_p$-C$^\ast$-algebra (Theorem \ref{thm:characterization_p_continuity}).  The estimates given by our $L_q$-property of rapid decay also enable us to show, using a quantum ``conjugation by generators'' technique, that these C$^\ast$-algebras admit a unique tracial state.  We point out, however, that the non-commutative nature of the spaces $L_p(\G)$ in the quantum setting make some of the required techniques quite different from the corresponding ones for discrete  groups.  In particular, we are required to make extensive use of complex interpolation techniques and some notions from the theory of operator spaces to eventually arrive at Theorem \ref{thm:exotic}.  

Using the above tools developed to study the $L_p$-C$^\ast$-algebras of $\F O_F$, we also investigate Van Daele and Wang's unimodular unitary free quantum groups $\F U_F$ \cite{VaWa}, and obtain similar results there (see Corollary \ref{cor:exoticU}).  

Along the way to proving the above theorems , we also derive several general  results on $D$-C$^\ast$-algebras of locally compact quantum groups which are of independent interest.  In particular, we characterize the Haagerup property of a  locally compact quantum group $\G$ in terms of its $C_0(\G)$-C$^\ast$-algebra (Theorem \ref{thm:HAP}).  We characterize the amenability of a unimodular discrete quantum group $\G$ in terms of the canonical quotient map from the universal C$^\ast$-algebra of $\G$ to $C^*_{L_p(\G)}(\G)$ being isometric for some finite $1  < p < \infty$ (Theorem \ref{thm:amenability}).   We also introduce and develop some basic properties of the $D$-Fourier and $D$-Fourier-Stieltjes algebras $A_D(\G)$ and $B_D(\G)$ associated with a locally compact quantum group $\G$.  These are natural ``$D$-analogues'' of the usual Fourier and Fourier-Stieltjes algebras for locally compact groups introduced in \cite{Ey}.  We note here that similar objects were considered in the context of ideal completions for locally compact groups in the work of Kaliszewski, Landstad and Quigg \cite{KaLaQu}.
          
This paper is organized as follows.  Section \ref{section:prelims} contains a brief review of some aspects of the theory of locally compact quantum groups and their Hilbert space representations.  Section \ref{section:dreps} introduces the notion of a $D$-representation of a locally compact group $\G$ and the corresponding $D$-C$^\ast$-algebra of $\G$, where $D$ is any linear subspace of $C_b(\G)$.  Here we introduce the $D$-Fourier(-Stieltjes) algebras, prove that $C^*_D(\G)$ carries the structure of a compact quantum group whenever $\G$ is discrete and $D$ is a subalgebra of $L_\infty(\G)$, and study the Haagerup property of locally compact quantum groups in terms of $C_0(\G)$-representations.  In Section \ref{section:unimodular}, we restrict our attention to unimodular discrete quantum groups and study the ideals $L_p(\G)$  and the corresponding $L_p$-C$^\ast$-algebras.  In the final Section \ref{section:applications}, we recall some facts about the (unimodular) free orthogonal/unitary quantum groups and then perform a detailed analysis of their $L_p$-C$^\ast$-algebras. 

\subsection*{Acknowledgements}  The authors wish to thank Quanhua Xu for fruitful conversations related to complex interpolation at an early stage of this work and to thank the referee for helpful suggestions.

\section{Preliminaries} \label{section:prelims}

In the following, we write $\otimes$ for the minimal tensor product of C$^\ast$-algebras or  tensor product of Hilbert spaces, $\overline{\otimes}$ for the spatial tensor product of von Neumann algebras and $\odot$ for the algebraic tensor product.  All inner products are taken to be conjugate-linear in the second variable.

\subsection{Locally compact quantum groups}

Let us first recall from  \cite{KuVa1} and \cite{KuVa2} that 
a  (von Neumann algebraic) \textit{locally compact quantum group} is 
a   quadruple $\G = (M,\Delta, h_L, h_R)$, where 
$M$ is a von Neumann algebra,  $\Delta:M \to M\overline{\otimes}M$ is a co-associative coproduct, i.e. a 
unital normal $\ast$-homomorphism such that 
\[
(\iota \otimes \Delta) \circ \Delta = (\Delta \otimes \iota) \circ \Delta,
\]
and $h_L$ and $h_R$ are normal faithful semifinite weights on $M$ such that 
\[ 
 (\iota  \otimes h_L)\Delta(x)  = h_L(x) 1 \quad \mbox{and}  \quad
 (h_R \otimes \iota)\Delta(x)  = h_R(x) {1}
\qquad (x\in M^+). 
\]
We call $h_L$ and $h_R$ the 
\textit{left Haar weight} and the \textit{right Haar weight} of $\G$, respectively, and 
we write  $L_\infty(\G) $  for the quantum group von Neumann algebra $M$.

Associated with  each  locally compact quantum group $\G$,  there is a \emph{reduced} quantum group 
 C*-algebra
 $C_0(\G)\subseteq L_\infty(\G)$  with coproduct  
  \[
 \Delta: x\in C_0(\G) \mapsto \Delta(x) \in M(C_0(\G)\otimes C_0(\G)) \subseteq L_\infty(\G) \overline {\otimes} L_\infty(\G).
 \]
Here we let  $M(B)$ denote the \emph{multiplier algebra } of a C*-algebra $B$.
Therefore, $(C_0(\G), \Delta)$ together with $h_L$ and $h_R$ restricted to $C_0(\G)$  is a   
C*-algebraic  locally compact quantum group.

Using the left Haar weight $h_L$, we can  apply the GNS construction to obtain an inner product  
\[
\langle {\Lambda(x) | \Lambda(y)} \rangle = h_L(y^* x)
\]
on  $ \mf N_{h_L}=  \{x \in L_\infty(\G) : h_L(x^*x) < \infty \}$, 
and thus obtain  a Hilbert space $L_2(\G)= L_2(\G, h_L)$.  Here, $\Lambda: \mf N_{h_L} \to L_2(\G)$ is the canonical injection.
The quantum group von Neumann algebra  $L_\infty(\G)$ is standardly represented on $L_2(\G)$ via  the unital normal
$\ast$-homomorphism $\pi : L_\infty(\G) \to \mc B(L_2(\G))$ satisfying  $\pi(x)\Lambda(y) = \Lambda(xy)$.
There exists a  (left) \textit{fundamental unitary operator } $W$ on $L_2(\G)\otimes L_2(\G)$, 
which satisfies the pentagonal relation
\[
W_{12} W_{13} W_{23} = W_{23} W_{12}.
\]
Here, and later,
we use the standard leg number notation: $W_{12} = W\otimes I, W_{13}
= \Sigma_{23} W_{12} \Sigma_{23}$ and $W_{23} = 1 \otimes W$, where
 $\Sigma: L_2(\G) \otimes L_2(\G)  \to L_2(\G) \otimes L_2(\G) $ is the flip map.
In this case, the coproduct  $\Delta$ on $L_\infty(\G)$ can be expressed as $\Delta(x) = W^*(1\otimes x)W$.

Let  $L_1(\G) = L_\infty(\G)_*$ be the predual of $L_\infty(\G)$.
Then the   pre-adjoint  of  $\Delta$  induces on $L_1(\G)$ a completely contractive Banach algebra product
\[
\star = \Delta_*: \omega \otimes \omega' \in L_1(\G) \widehat \otimes L_1(\G) \mapsto \omega \star \omega' = (\omega \otimes \omega' ) \Delta \in L_1(\G),
\]
where we let  $\widehat \otimes$ denote the operator space projective tensor product.
Using the Haar weights, one can construct an antipode $S$ on $L_\infty(\G)$ satisfying $(S \otimes \iota)W = W^*$. Since, in general, $S$ is 
unbounded on $L_\infty(\G)$, we  can not use $S$ to define an involution on $L_1(\G)$.
However, we can consider a dense subalgebra $L_1^\sharp (\G)$ of $L_1(\G)$, which is defined to be 
the collection of all  $\omega\in L_1(\G)$ such that there exists $\omega^\sharp \in L_1(\G)$ with
$\ip{\omega^\sharp}{x} = \overline{ \ip{\omega}{S(x)^*} }$ for each $x\in \text{Dom}(S)$.
It is known from  \cite{Ku} and \cite[Section~2]{KuVa2} that 
 $L_1^\sharp(\G)$ is an involutive Banach algebra with  involution  $\omega\mapsto\omega^\sharp$
 and norm $\|\omega\|_\sharp = \mbox{max}\{\|\omega\|, \|\omega^\sharp\|\}$.

\subsection{Representations of locally compact quantum groups} 

A \emph{representation} $(\pi, H)$ of a locally compact quantum group $\G$ is  a non-degenerate completely contractive homomorphism  $\pi: L_1(\G) \to \mc B(H)$   whose restriction to $L_1^\sharp(\G)$  is a  $\ast$-homomorphism.
It is shown in \cite[Corollary 2.13]{Ku}  that each representation $(\pi, H)$ of $\G$ corresponds uniquely to a unitary operator $U_\pi$ 
 in $M(C_0(\G) \otimes \mc K(H))\subseteq L_\infty (\G) \overline{\otimes} \mc B(H)$ such that 
 \[
 (\Delta \otimes \iota)  U_\pi =  U_{\pi, 13}  U_{\pi, 23}.
 \]
The correspondence  is given by 
\[
\pi(\omega)  = (\omega \otimes \iota) U_\pi \in  \mc B(H) \qquad (\omega \in L_1(\G)).
\]  
We call $U_\pi $ the \emph{unitary representation} of $\G$ associated with $\pi$. 

Given two representations $(\pi,H_\pi) $ and $ (\sigma ,H_\sigma)$, we can obtain new representations by forming their tensor product and direct sum.  The unitary operator
\[ 
U_{\pi \otop \sigma}  :=  U_{\pi,12} U_{\sigma,13} \in M(C_0(\G) \otimes\mc K(H_\pi \otimes H_\sigma))
\]  
determines a representation 
\[
\pi \otop  \sigma: \omega \in  L_1(\G) \mapsto \mc (\omega \otimes \iota)U_{\pi \otop \sigma}  \in \mc B(H_\pi \otimes H_\sigma).
\]
We call the representation $({\pi \otop \sigma}, H_\pi \otimes H_\sigma)$  the  \textit{tensor product}  of $\pi$ and $\sigma$.
The \emph{direct sum}  $(\pi\oplus \sigma, H_\pi \oplus H_\sigma)$ of $\pi$ and $\sigma$ is given by the representation
\[
\pi\oplus \sigma : \omega\in L_1(\G) \to \pi(\omega) \oplus \sigma(\omega)\in \mc B(H_\pi \oplus H_\sigma).
\]
We often use the notion $U_\pi \oplus U_\sigma$ for the corresponding unitary operator $U_{\pi \oplus \sigma}$.

For representations $(\pi,H)$ of $\G$,  one of course has all of the usual notions of (ir)reducibility, intertwiners, unitary equivalence etc. that can be considered for representations of general Banach algebras. 
 
 The \emph{left regular representation} is defined by 
 \[
 \lambda : \omega \in L_1(\G) \mapsto (\omega \otimes \iota )W\in \mc B(L_2(\G)).
 \]
This is an injective  completely contractive  homomorphism from $L_1(\G)$ into $\mc B(L_2(\G))$ such that 
$\lambda(\omega^\sharp) = \lambda(\omega)^*$ for all $\omega\in L_1^\sharp(\G)$.
Therefore,  $\lambda$ is a $*$-homomorphism when restricted to $L_1^\sharp(\G)$.
The   norm closure of $\lambda(L_1(\G))$ gives us the reduced   quantum group  C$^\ast$-algebra $C_0(\hat \G)$  and the $\sigma$-weak closure of
$\lambda(L_1(\G))$ gives us the quantum group  von Neumann algebra $L_\infty(\hat \G)$.  
 There is a coproduct on $L_\infty(\hat \G)$ given by 
 \[
\hat{\Delta}:  \hat x \in L_\infty(\hat \G) \to \hat\Delta(\hat x) = \hat W^*(1\otimes \hat x) \hat W \in 
L_\infty(\hat \G)\overline{\otimes} L_\infty(\hat \G),
\]
 where   $\hat W = \Sigma W^*\Sigma$.
We can find 
Haar weights $\hat h_L$ and $  \hat h_R$ to turn $\hat \G = (L_\infty(\hat \G) , \hat \Delta, \hat h_L, \hat h_R)$
 into a locally compact quantum group -- the dual quantum group to $\G$.  
 Repeating  this argument for the dual quantum group $\hat \G$ , we get  the left regular representation 
 \[
 \hat \lambda : \hat\omega \in L_1(\hat \G) \mapsto (\hat \omega \otimes \iota )\hat W=  (\iota \otimes \hat \omega )W^*
 \in \mc B(L_2(\G)).
 \]
 It turns out that   $C_0(\G)$ and   $L_\infty(\G)$ are just 
  the norm and $\sigma$-weak closure  of  $\hat \lambda(L_1(\hat \G))$ in $\mc B(L_2(\G))$, respectively.
 This gives us   the  quantum group version of Pontryagin duality $\hat{\hat{\G}}= \G$.
 

There is  a \emph{universal representation} 
 $\pi_u: L_1(\G) \to \mc B(H_u)$ and
 we   obtain   the  {\it   universal quantum group C$^\ast$-algebra}   $C_u(\hat \G) = \overline{\pi_u(L_1(\G))}^{\|\cdot\|}$.  
 By the universal property,   every representation $\pi: L_1(\G) \to \mc B(H)$  uniquely corresponds to a surjective 
  $\ast$-homomorphism $\hat \pi$ from $C_u(\hat \G)$
  onto the C*-algebra $C_\pi(\hat \G) = \overline{\pi(L_1(\G))}^{\|\cdot\|}$ such that $\pi = \hat \pi \circ \pi_u .$
  In particular,   the left regular representation  $(\lambda, L_2(\G))$ uniquely determines a surjective $\ast$-homomorphism $\hat \pi_\lambda$ from $C_u(\hat \G)$ onto  $ C_0(\hat \G)$.  Considering the duality, we can obtain the universal quantum group C$^\ast$-algebra $C_u(\G) $ and we denote by $ \pi_{\hat \lambda} : C_u(\G) \to C_0(\G)$ the canonical surjective $\ast$-homomorphism.  As shown in \cite{Ku}, $C_u(\G)$ admits a coproduct $\Delta_u:C_u(\G) \to M(C_u(\G) \otimes C_u(\G))$ which turns $(C_u(\G), \Delta_u)$ into a C$^\ast$-algebraic locally compact quantum group with left and right Haar weights given by $h_L \circ \pi_{\hat \lambda}$ and $h_R \circ \pi_{\hat \lambda}$, respectively.    
A locally compact quantum group $\G$ is called \textit{co-amenable} if $\pi_{\hat \lambda}$ is an isomorphism, and $\G$ is called \textit{amenable} if
$L_\infty(\G)$ admits a \textit{left invariant mean}, i.e.  there is a state $m \in L_\infty(\G)^*$ such that 
\[m\big((\omega \otimes \iota)\Delta(x)\big) = \omega(1) m(x)  \qquad (\omega \in L_1(\G), \ x \in L_\infty(\G)).
\]   

It is shown in \cite{Ku} that the fundamental unitary $W$ of $\G$ admits a ``semi-universal'' version $\text{\reflectbox{$\W$}}\in M(C_0(\G) \otimes C_u(\hat \G))$ which has the property that for each representation $\pi: L_1(\G) \to \mc B(H)$, the corresponding unitary operator $U_\pi$ can be expressed as   
\[(\iota \otimes \hat \pi)\text{\reflectbox{$\W$}} = U_\pi.
\] 
Moreover, $\text{\reflectbox{$\W$}}$ satisfies  the following relations
\begin{align} \label{bichar_relation}
(\Delta \otimes \iota) \text{\reflectbox{$\W$}} = \text{\reflectbox{$\W$}}_{13}\text{\reflectbox{$\W$}}_{23} ~\mbox{and} ~ 
(\iota \otimes \hat \Delta_u) \text{\reflectbox{$\W$}} = \text{\reflectbox{$\W$}}_{13} \text{\reflectbox{$\W$}}_{12}.
\end{align} 

  
\subsection{Compact and discrete quantum groups}  A locally compact quantum group is called 
\textit{compact} if $C_0(\G)$ is unital, or equivalently $h_L = h_R$ is a bi-invariant state (after an appropriate normalization).  In this case, we will write $C(\G)$ for $C_0(\G)$.
 We say $\G$ is \textit{discrete} if $\hat \G$ is a compact quantum group, or equivalently, if $L_1(\G)$ is unital.  Note that a discrete quantum group is always co-amenable, and compact quantum groups are always amenable.
In particular, we always have $C_u(\G) = C_0(\G)$ for discrete quantum groups.

Let $\G$ be a discrete quantum group.  Denote by $\Irr(\hat \G)$ the collection of equivalence classes of finite dimensional irreducible representations of $\hat \G$.  For each $\alpha \in \Irr(\hat  \G)$, select a representative unitary representation $(\hat U^\alpha,H_\alpha)$.  Then $\hat U^\alpha$ can be identified with a unitary matrix 
\[
\hat U^\alpha = [\hat u_{ij}^\alpha] \in M_{d_\alpha}( C(\hat \G)),
\]
where $d_\alpha = \dim H_\alpha$.  The linear subspace $\Pol(\hat \G) \subseteq C(\hat \G)$ spanned by $\{\hat u_{ij}^\alpha : \alpha \in \Irr(\hat \G), \ 1 \le i,j \le d_\alpha\}$ is a dense Hopf-$\ast$-subalgebra  of $C(\hat \G)$.  The  coproduct on $\Pol(\hat \G)$ is given by  restricting $\hat \Delta $ to $\Pol(\hat \G)$, and we have 
\[
\hat \Delta  (\hat u^\alpha_{ij} ) = \sum_{k=1}^{d_\alpha} \hat u^\alpha_{ik} \otimes \hat u^\alpha_{kj}.
\]
We call $\Pol(\hat \G)$ the \textit{algebra of polynomial functions on $\hat \G$}.  

The Haar state $\hat h =  \hat h_L = \hat h_R$ is always faithful when restricted
to $\Pol(\hat \G)$, and  $C_u(\hat \G)$ can be identified with the universal enveloping C$^\ast$-algebra of $\Pol(\hat \G)$.  This allows us to regard $\Pol(\hat \G)$ as a dense $\ast$-subalgebra of $C_u(\hat \G)$, and the semi-universal fundamental unitary of 
the discrete quantum group $\G$ is given by 
\begin{align} \label{bichar_compact} \text{\reflectbox{$\W$}} = \bigoplus_{\alpha \in \Irr(\hat \G)} (\hat U^\alpha)^* =  \bigoplus_{\alpha \in \Irr(\hat \G)} \sum_{1 \le i,j \le d_\alpha} e_{ij}^{\alpha}  \otimes \hat  u_{ji}^{\alpha \ast}  \in M(C_0( \G) \otimes C_u(\hat \G)). 
\end{align}  
In general,  for any discrete quantum group $\G$, we have
\[C_0(\G) =c_0 - \bigoplus_{\alpha \in \Irr(\hat \G)} \mc B(H_\alpha)  \quad \mbox{and} \quad L_\infty(\G) = \prod_{\alpha \in \Irr(\hat \G)} \mc B(H_\alpha).
\]  
Denote by $C_c(\G)$ the (algebraic) direct sum $\bigoplus_{\alpha \in \Irr(\hat \G)} \mc B(H_\alpha)$.  Then 
$C_c(\G)$  forms a common core for the Haar weights $h_L$ and $h_R$.   

For each $\alpha \in \Irr(\hat \G)$, denote by $p_\alpha \in L_\infty(\G)$ the minimal central projection whose support is $\mc B(H_\alpha)$.  
Let $\Tr_{\alpha}$
be the canonical trace  on  $\mc B(H_\alpha)$ (with $\Tr_{\alpha}(1) = d_\alpha$).
One can then find positive invertible matrices $Q_\alpha \in \mc B(H_\alpha)$ which satisfy 
\[
m_\alpha =\Tr_\alpha(Q_\alpha) = \Tr_\alpha(Q_\alpha^{-1}) \quad \mbox{and} \quad 
\Delta(Q_\alpha) = Q_\alpha \otimes Q_\alpha.
\]  The left and right  Haar weights on $\G$ are given by the formulas 
 \begin{align} \label{eqn:weights}
 h_L(p_\alpha x)= m_\alpha \Tr_\alpha(Q_\alpha^{-1} p_\alpha x)  \quad \mbox{and} 
  \qquad h_R(p_\alpha x) = m_\alpha \Tr_\alpha(Q_\alpha p_\alpha x). 
\end{align}  See \cite[Equations (2.12)-(2.13)]{PoWo}.
Using the left Haar weight $h_L$ (and the fact that $L_\infty(\G)$ is a direct product of full matrix algebras), we have 
\[
L_2(\G) = \ell_2- \bigoplus_{\alpha\in \Irr(\hat \G)}  L_2(\mc B(H_\alpha), h_{L}|_{\mc B(H_\alpha)})
= {\frak N}_{h_L}. 
\]

As in the case of discrete groups, for each $\omega \in L_1(\G)$, we can find a unique $x \in L_\infty(\G)$ such that $\omega = \omega_x$, where we define $\omega_x  (y) = h_L(yx)$ for each $y \in L_\infty(\G)$.  Indeed, for each $\alpha \in \Irr(\hat \G)$ let $\tilde{x}_\alpha \in \mc B(H_\alpha)$ be the unique element such that $\omega|_{\mc B(H_\alpha)} = \Tr_\alpha(\cdot \ \tilde{x}_\alpha).$  Then one readily computes \[\omega = \omega_x \quad \text{where} \quad x = \Big(\frac{\tilde{x}_\alpha Q_\alpha}{\Tr_\alpha(Q_\alpha)} \Big)_{\alpha \in \Irr(\hat \G)}\in L_\infty(\G).\]  Moreover, we have $\|x\|_{L_\infty(\G)} \le \|\omega\|_{L_1(\G)}$.  This yields a linear identification of $L_1(\G)$ with the subspace \[{\frak M}_{h_L} := \Big\{x \in L_\infty(\G): \omega_x \in L_1(\G)\Big\} \subseteq L_\infty(\G),\]  
Using the above identification, we can transfer the convolution product on $L_1(\G)$ to ${\frak M}_{h_L}$.  Indeed, given $x, y\in {\frak M}_{h_L} $, there exists a unique element $z \in {\frak M}_{h_L} $ such that $\omega_z = \omega_x \star \omega_y$.
We write  $z =  x \star  y$ and call $z$ the \textit{convolution product } of $x$ and $y$.
Since $C_c(\G) \subseteq  {\frak M}_{h_L} $, 
we can naturally identify  $C_c(\G)$ with  a subspace ${\mathcal A}(\G) = \{ \omega_x : x \in C_c(\G)\} $ in $L_1(\G)$. 
In fact,   ${\mathcal A}(\G)$ is a norm dense Hopf  $\ast$-subalgebra of $L_1(\G)$ (see \cite {EfRu94}  and \cite {PoWo}).
Therefore, we can conclude that  the  convolution product is closed on $C_c(\G)$, i.e for any $x, y\in C_c(\G)$, we have $z = x \star y \in C_c(\G)$.  The above convolution product for $C_c(\G)$ was also considered by Vergnioux in \cite[Section 4.2]{Ve}.  We note that Vergnioux attributes the above convolution product to Podl\`es and Woronowicz \cite{PoWo}, but the reader should be aware that the convolution product considered here and in \cite{Ve} does not appear explicitly in \cite{PoWo}.  Indeed, \cite[Section 2]{PoWo} considers the \textit{dual} setting and constructs a convolution product on the Hopf $\ast$-algebra $\Pol(\hat \G)$.

\section{$D$-representations and associated C$^\ast$-algebras} \label{section:dreps}

\subsection{Coefficient functions of representations} 
Let $\pi: L_1(\G) \to \mc B(H)$ be a representation and let $(U_\pi, H)$ be the associated unitary representation.
For each $\xi, \eta \in H$, we let $\omega_{\xi, \eta} \in \mc B(H)_*$ denote
 the $\sigma$-weakly continuous linear functional $x \mapsto \langle x\xi | \eta \rangle$ and 
we call 
\[
\varphi^\pi_{\xi,\eta} = ( \iota \otimes \omega_{\xi,\eta})U_\pi =
 ( \iota \otimes \omega_{\xi,\eta}\circ \hat \pi)  \text{\reflectbox{$\W$}}   \in C_b(\G) \subseteq L_\infty(\G) 
\] 
a \textit{coefficient function} of the representation  $\pi$.  Equivalently, $\varphi^\pi_{\xi,\eta}$ is determined by the dual pairing
\[\langle \omega, \varphi^\pi_{\xi,\eta} \rangle = \langle \pi(\omega)\xi|\eta \rangle \qquad (\omega \in L_1(\G)). \]  
One can of course slice representations by more general bounded linear functionals 
 $\omega \in \mc B(H)_*$ and we also call $\varphi^\pi_{\omega} =( \iota \otimes \omega)(U_\pi) \in C_b(\G) \subseteq 
 L_\infty(\G)$ a coefficient function of $\pi$.  

Despite being unbounded in general, the antipode $S$ of $\G$ behaves quite nicely with respect to coefficient functions of representations of $\G$.  Indeed, for any unitary representation $(U_\pi,H)$ of $\G$, one has $U_\pi^* = (S \otimes \iota)U_\pi$, and consequently for any pair of vectors $\xi,\eta \in H$, $\varphi^\pi_{\xi,\eta} \in \text{Dom}(S)$ and $S(\varphi^\pi_{\xi,\eta}) = (\varphi^\pi_{\eta,\xi})^*$.  See for example 
 \cite[Proposition 4.4]{BrDaSa}. 

An important notion in this paper will be that of a positive definite function on $\G$.  

\begin{defn}
An element $y\in C_b(\G)$ is called a (\emph{norm one}) \emph{positive definite function}  if there exist a  representation $(\pi, H)$ and a (unit) vector $\xi\in H$ such that $y = \varphi^\pi_{\xi, \xi}$.  
\end{defn}

\begin{rem} \label{rem:pd}
In \cite{DaSa}, Daws and Salmi considered various quantum group generalizations of the notion of a positive definite function on a group.  The reader should be warned that our notion of positive definiteness differs from the one in \cite{DaSa}.  More precisely, an element $y \in C_b(\G)$ is a positive definite function in our sense if and only if $y^*$ is a \textit{Fourier-Stieltjes transform of a positive measure} in the sense of \cite{DaSa}.
\end{rem}

The following lemma is a straightforward consequence of the definitions.

\begin{lem} \label{lem:products_coefficients}
Let $(\pi,H_\pi)$ and $  (\sigma , H_\sigma)$  be two representations of $\G$.  Then we have 
\[
\varphi^{\pi \otop \sigma}_{\omega \otimes \omega'}= \varphi^{\pi}_{\omega}\varphi^{\sigma}_{\omega'} 
~\mbox{and} ~ \varphi^{\pi \oplus \sigma}_{\omega \oplus \omega'}= \varphi^{\pi}_{\omega} + \varphi^{\sigma}_{\omega'} 
\qquad (\omega \in \mc B(H_\pi)_*, \omega' \in \mc B(H_\sigma)_*).\]
\end{lem} 


Since $L_\infty(\G)$ is the dual space of $L_1(\G)$, the 
 product on $L_1(\G)$ induces a natural  $L_1(\G)$-bimodule action  on $L_\infty(\G)$ given by 
\[\omega \star x = (\iota \otimes \omega) \Delta(x)  \quad \mbox{and} \quad x \star \omega = (\omega \otimes \iota ) \Delta(x)   \qquad (\omega \in L_1(\G), \ x \in L_\infty(\G)).
\]

\begin{lem} \label{lem:leftright_action}
For any representation $(\pi,  H)$,  $\omega \in L_1^\sharp( \G)$ and  $\ \xi, \eta \in H$,   we have 
\[
\varphi^\pi_{\pi(\omega)\xi, \eta} = \omega \star \varphi^\pi_{\xi, \eta} \quad \text{and} \quad \varphi^\pi_{\xi, \pi(\omega)\eta} =  \varphi^\pi_{\xi, \eta}\star \omega^\sharp.
\]  
\end{lem}  

\begin{proof}Let $\omega' \in L_1(\G)$.  Then 
\begin{align*}\langle \omega', \varphi^\pi_{\pi(\omega)\xi, \eta}\rangle& = \langle \pi(\omega')\pi(\omega)\xi| \eta \rangle= \langle \omega' \star \omega, \varphi^\pi_{\xi, \eta}\rangle = \langle \omega', \omega\star \varphi^\pi_{\xi, \eta} \rangle. 
  \end{align*}
This proves the first equality.  The second equality is proved in the same fashion:
\[\langle \omega', \varphi^\pi_{\xi, \pi(\omega)\eta}\rangle = \langle \pi(\omega)^*\pi(\omega')\xi| \eta \rangle= \langle \omega^\sharp \star \omega', \varphi^\pi_{\xi, \eta}\rangle = \langle \omega',  \varphi^\pi_{\xi, \eta} \star \omega^\sharp \rangle.\]
\end{proof}

\subsection{$D$-Representations}
Let us begin with  the definition of $D$-representations.

\begin{defn}
Let $\G$ be a  locally compact quantum group and let 
$D$ be a (not necessarily closed) linear subspace of $C_b(\G)$.  
A representation $(\pi, H)$ of $\G$  is called a \textit{$D$-representation} if there is a dense subspace $H_0 \subseteq H$ such that for all pairs $\xi, \eta \in H_0$,  the coefficient function $\varphi^\pi_{\xi, \eta}$ belongs to  $D$.  
 \end{defn}

We can easily obtain the following functorial properties of $D$-representations.

 \begin{lem} \label{lemproduct} 
 Let $\G$ be a locally compact quantum group.
 \begin{itemize}
 \item [(1)] The direct sum of a  family  of $D$-representations is a $D$-representation.
 \item [(2)] If $D$ is a subalgebra of $C_b(\G)$, the tensor product of two $D$-representations is a $D$-representation.
 \item[(3)] If $D$ is a two-sided ideal in $C_b(\G)$, the tensor product of a $D$-representation with an arbitrary representation is again a $D$-representation.
 \end{itemize}
\end{lem}
\begin{proof} We provide a proof for (1) and (3).
The argument for (2) is similar.
Let  $\{(\pi_\alpha, H_\alpha)\}$  be a family of $D$-representations with norm dense subspaces  $H_{\alpha, 0}\subseteq H_\alpha$.
It is known from  Lemma \ref {lem:products_coefficients} that  $(\oplus_\alpha \pi_\alpha, \oplus _\alpha H_\alpha)$ is again a representation
with the associated unitary operator given by  $\oplus_\alpha U_\alpha$.
In this case, $H_0 = \{\xi = (\xi_\alpha):  \xi_\alpha \in H_{\alpha, 0}  \mbox{ and  finitely many } 
~ \xi_\alpha \neq 0\}$ is a norm dense subspace of 
$ \oplus _\alpha H_\alpha$ such that 
\[
\varphi ^{\oplus_\alpha \pi_\alpha}_{\xi, \eta} = \sum_\alpha \varphi^{\pi_\alpha}_{\xi_\alpha, \eta_\alpha} \in D
\]
for all $\xi = (\xi_\alpha), \eta = (\eta_\alpha)  \in H_0$.
This  shows that $\oplus_\alpha \pi_\alpha$ is  a $D$-representation.  

Let $D$ be a two-sided ideal, $(\pi,H)$ a $D$-representation with norm dense subspace $H_0$, and $(\sigma, K)$ another representation.  Then Lemma \ref{lem:products_coefficients} implies that $(\pi \otop \sigma, H \otimes K)$ and $(\sigma \otop \pi, K \otimes H)$ are $D$-representations with dense subspaces $H_0 \odot K$ and $K \odot H_0$, respectively.
\end{proof}


Clearly every representation of $\G$ is an $C_b(\G)$-representation.  Using the language of \cite{DaFiSkWh}, the $C_0(\G)$-representations are precisely the  \textit{mixing} representations of $\G$ and arise in the study of the Haagerup property.  We will discuss this further in Section \ref{section:HAP}.
When $\G$ is a discrete quantum group, we get $L_\infty(\G) = C_b(\G)$.  In this setting another important space to consider is the two-sided ideal $D=C_c(\G)$ in
$L_\infty(\G)$ and the corresponding class of $C_c(\G)$-representations.  The following proposition shows that for discrete quantum groups, the left regular representation $(\lambda, L_2(\G))$ is the prototypical example of a $C_c(\G)$-representation.  

\begin{prop} \label{prop:regular_is_C00}  If  $\G$ is a discrete quantum group, 
the left regular representation $(\lambda ,L_2(\G))$ is a $C_c(\G)$-representation.
\end{prop}

\begin{proof}
Since  $C_c(\G) $ is a dense subspace of $  L_2(\G)$, 
it suffices to show that \[\varphi^\lambda_{\Lambda(x), \Lambda(y)} \in C_c(\G) \qquad (x,y \in C_c(\G)).\] 
Let us recall  from  equation (8.1) in \cite{KuVa1} that   
\begin{align} \label{eq:KuVa}\varphi^\lambda_{\Lambda(x), \Lambda(y)} = (\iota \otimes \omega_{\Lambda(x), \Lambda(y)}) W = (\iota \otimes h_L)(\Delta(y^*)(1 \otimes x)) \qquad (x,y \in \mathfrak N_{h_L}).
\end{align}  
Since   $\Delta ( C_c(\G))(1 \otimes C_c(\G)) \subseteq C_c(\G)\odot C_c(\G)$ 
(see \cite[Corollary 6.5]{EfRu94}), we have $\varphi^\lambda_{\Lambda(x), \Lambda(y)} \in C_c(\G) $  for all 
$x, y\in C_c(\G) \subseteq  L_2(\G)$
\end{proof}

\begin{rem} \label{rem:Nh*}
For general (non-discrete/non-compact) locally compact quantum groups, there is no good notion for the space $C_c(\G)$.
 In this case, we can replace $C_c(\G)$ in Proposition \ref {prop:regular_is_C00}  with the space $\mathfrak N^*_{h_L} \cap C_b(\G)$ and show that the left regular representation $(\lambda, L_2(\G))$ is an $\mathfrak N^*_{h_L}  \cap C_b(\G)$-representation.
 To see this, let  us  recall from  \cite[Section 3]{BuMe09} that a  representation $\pi: L_1(\G) \to \mc B(H)$ is \emph{square integrable} if
for each $\alpha$ in a dense subspace of $H$ and every $\beta \in H$, 
the bounded linear functional $C^\pi_{\alpha, \beta} = \omega_{\alpha, \beta} \circ \hat \pi \in C_u(\hat \G)^*$ is square integrable.  I.e., the map
\[
 \hat x \in {\rm Dom}(\hat \Lambda_u) \mapsto C^\pi_{\alpha, \beta} (\hat x)  \in \mathbb C 
  \]
 extends  to a bounded linear functional on $L_2(\hat \G)$. 
 Let us recall that there is a dense subspace 
 $ {\mathcal I} = \{\omega \in L_1(\G): ~\mbox {there exists } ~M\ge 0 ~\mbox{such that}~ |\omega(x^*)| \le M \|\Lambda(x)\|: x\in {\frak N}_{h_L}\}$ in $L_1(\G)$, and  each $\omega\in {\mathcal I}$ determines a unique
 element $\xi(\omega)\in L_2(\G)$ such that   $\langle \xi(\omega) |  \Lambda(x)\rangle = \omega(x^*)$ for all $x \in {\frak N}_{h_L}$ (see  \cite[Section 8]{KuVa1}).
Then with respect to the dual Haar weight $\hat h_{L}$, we have the isometric identification
\[
\xi(\omega) \in L_2(\G)  \to  \hat \Lambda (\lambda(\omega)) = \hat \Lambda_u(\pi_u(\omega)) \in L_2(\hat \G)   \qquad (\omega \in {\mathcal I}).
\]
 Since
 \[
C^\pi_{\alpha, \beta} (\pi_u(\omega)) = (\omega \otimes \omega_{\alpha, \beta}\circ \hat \pi)\text{\reflectbox{$\W$}}
= \langle \omega, \varphi^\pi_{\alpha, \beta}\rangle = \langle \xi(\omega) |  \Lambda((\varphi^{\pi }_{\alpha, \beta})^*)\rangle
 \qquad (\omega \in {\mathcal I}), 
\]
we conclude from  \cite[Proposition 1.11.26]{Vaes} that  $C^\pi_{\alpha, \beta}$ is square integrable if and only if  the coefficient function  $\varphi^\pi_{\alpha, \beta} $ is contained in $ \mathfrak N^*_{h_L}  \cap C_b(\G)$.  Thus, whenever $\pi$ is a square integrable representation, it is also a $ \mathfrak N^*_{h_L}  \cap C_b(\G)$-representation. (Note that the converse is not being claimed!)  Since the left regular representation $(\lambda, L_2(\G))$ is square integrable \cite [Lemma 3.7]{BuMe09}, it is an $\mathfrak N^*_{h_L}  \cap C_b(\G)$-representation
 \end{rem} 
  
The following result is the discrete quantum group analogue of the fact that any square summable function on a discrete group can be realized as a coefficient function of the left regular representation.

\begin{prop} \label{prop:C_00inA(G)}  If  $\G$ is a discrete quantum group, 
 every element $y \in \mathfrak N_{h_L}^*$ is a
coefficient function of the left regular representation $(\lambda, L_2(\G))$.
If, in addition,  $y $ is a norm one positive definite function, then there is a unit vector $\xi\in L_2(\G)$ such that 
$y= \varphi^\lambda_{\xi, \xi}$.
\end{prop}
\begin{proof}

Let $p_{0} \in L_\infty(\G)$ be the minimal central projection associated with the  equivalence class of the trivial representation of $\hat \G$.  Since 
$\omega_{p_{0}}$ is the unit of $L_1(\G)$, equation \eqref{eq:KuVa} gives  
\begin{align*}   \langle \omega, y  \rangle &=\langle \omega\star \omega_{p_{0}}, y  \rangle
= \langle \omega, (\iota \otimes h_L)(\Delta(y)(1 \otimes p_{0}))  \rangle 
= \langle \omega, \varphi^\lambda_{\Lambda(p_{0}), \Lambda(y^*)} \rangle   \qquad (\omega \in L_1(\G)).
\end{align*}
Therefore $y = \varphi^\lambda_{\Lambda(p_{0}), \Lambda(y^*)}$ is a coefficient function of the left regular representation.  In particular, $y$ determines a normal linear functional on $L_\infty(\hat \G) = \lambda(L_1^\sharp(\G))''$ such that
\[
\lambda(\omega) \mapsto \langle \omega, y \rangle = \langle  \lambda(\omega) \Lambda(p_{0}) |   \Lambda(y^*) \rangle
\qquad   (\omega\in L_1^\sharp(\G)).
\]  If $y = \varphi^\pi_{\eta, \eta}$ is also norm one positive  definite function, then the above linear functional is a normal state since  
\[
\langle \omega^\sharp \star \omega, y \rangle =  \langle \pi(\omega^\sharp \star \omega) \eta | \eta\rangle \ge 0 \qquad  (\omega\in L_1^\sharp(\G)).
\]
Since $L_\infty(\hat \G)$ is in standard position on $L_2(\G)$, it follows that  $y = \varphi^\lambda_{\xi, \xi}$ for some unit vector $\xi\in L_2(\G)$.   
\end{proof}


\subsection{The $D$-C$^*$-algebra of $\G$}
Let $\G$ be a locally compact quantum group and let $D$ be a linear subspace  of $C_b(\G)$.
We can define a C$^\ast$-semi-norm 
\[\|\omega\|_{D}  = \sup\{\|\pi(\omega)\|: \pi \text{ is a $D$-representation of $\G$} \} \] 
 on $L_1 (\G)$.
 The (non-)degeneracy of the semi-norm $\|\cdot\|_{D}$ depends on the particular structure of the subspace $D$.  For general locally compact quantum groups $\G$, the inclusion $\mathfrak N_{h_L} ^*\cap C_b(\G) \subseteq D$ suffices to ensure that $\|\cdot\|_{D}$ is non-degenerate because the left regular representation $(\lambda, L_2(\G))$ is always a faithful $D$-representation (see Remark \ref {rem:Nh*}).  
When $\G$ is a discrete quantum group, a weaker sufficient condition for the non-degeneracy of $\|\cdot\|_D$ is that $C_c(\G) \subseteq D$.  This follows because the left regular representation is a faithful $C_c(\G)$-representation by Proposition \ref{prop:regular_is_C00}.  Since we are generally interested in linear subspaces $D \subseteq C_b(\G)$ for which $\|\cdot\|_D$ is a C$^\ast$-norm on $L_1(\G)$ that dominates the $C(\hat \G)$-norm, this leads us to make the following assumption throughout the rest of the paper.

\begin{assumption} \label{assumption}  In this paper, we only consider linear subspaces $D \subseteq C_b(\G)$ for which $(\lambda, L_2(\G))$ is a $D$-representation.  In particular, we always assume $\|\cdot\|_D$ is a C$^\ast$-norm on $L_1(\G)$.
\end{assumption}  

Under the above assumption, we define the \textit{$D$-C$^\ast$-algebra} of $\G$,
 \[
 C^*_{D}(\G) = \overline{L_1 (\G)}^{\|\cdot\|_D} =  \overline{L_1^\sharp (\G)}^{\|\cdot\|_D} 
 \]
to be the norm closure of $L_1(\G)$ (respectively, the closure of $L_1^\sharp(\G)$)
with respect to the norm $\|\cdot\|_D$.  When $D= C_b(\G)$, we will simply write $C^*_\infty(\G)$ instead of $C^*_{C_b(\G)}(\G)$.

\begin{rem} \label{rem:discrete/compact}For a discrete quantum group $\G$, we can equivalently consider $C^*_{D}(\G)$ as a C$^\ast$-completion of the algebra $\Pol(\hat \G)$ of polynomial functions on $\hat \G$.  Indeed, since there is a natural $\ast$-isomorphic identification between   
${\mathcal A}(\G) \subset L_1^\sharp(\G)$ and $\Pol(\hat \G)$, we see that each representation $(\pi,H)$ of $\G$ uniquely corresponds to a $\ast$-representation $(\hat \pi,H)$ of $\Pol(\hat \G)$.
We can then equivalently define a C$^\ast$-norm $\|\cdot\|_D$ on $\Pol(\hat \G)$ by setting
 \[\|\hat a\|_D = \sup\{\|\hat \pi(\hat a)\|: (\pi,H_\pi) \text{ is a $D$-representation of $\G$}\} \qquad (\hat a \in \Pol(\hat \G)).\] 
If we let $C_D(\hat \G) = \overline{\Pol(\hat \G)}^{\|\cdot\|_{D}},$ then we obtain 
a canonical $\ast$-isomorphism $C_D(\hat \G) \cong C^*_D(\G)$ extending the isomorphism $\Pol(\hat \G) \cong \mc A(\G)$.  In the remainder of the paper, we will regard $C^*_D(\G)$ and $C_D(\hat \G)$ as the same object and simply refer to it as \textit{the} $D$-C$^\ast$-algebra of the discrete quantum group $\G$. 
\end{rem}

The following are  some expected properties of $D$-C$^\ast$-algebras.

\begin{prop}\label{equiv2} Let $\G$ be a locally compact quantum group.
\begin{enumerate}
\item \label{ue} Every $D$-representation $(\pi,H)$ of $\G$ extends uniquely to a $\ast$-homomorphism $\pi:C^*_D(\G) \to \mc B(H)$.
\item \label{faithfulDrep}    Every $D$-C$^\ast$-algebra $C_D^*(\G)$ admits a \textit{faithful} $\ast$-representation $\pi_D:C_D^*(\G) \to \mc B(H_D)$ whose restriction to $L_1(\G)$ is a $D$-representation.
\item \label{ideals_inclusion} If $D_1 \subseteq D_2$ are subspaces in $C_b(\G)$, then there exists a unique surjective $\ast$-homomorphism $\sigma:C_{D_2}^*(\G) \to C_{D_1}^*(\G)$ extending the identity map on $L_1 (\G)$. 
\item \label{full} The identity map on $L_1(\G)$ extends to a $\ast$-isomorphism $C_{\infty}^*(\G) \cong C_u(\hat \G)$.
\item \label{unimodular} If $\G$ is discrete and $C_c(\G) \subseteq D \subseteq {\frak N}_{h_L} ^*$,  then the identity map on $L_1 (\G)$ extends to a $\ast$-isomorphism  $C_D^*(\G) \cong C(\hat \G)$.
\end{enumerate}
\end{prop}

\begin{proof}
(\ref{ue}). Obvious. 

(\ref{faithfulDrep}).  It is routine to check that the set of $D$-representations of $\G$ separates the points in $C^*_D(\G)$.  Therefore, for each $0 \ne x \in C^*_D(\G)$ we can find a $D$-representation $(\pi_x,H_x)$ of $\G$ such $\pi_x(x) \ne 0$.  Then the direct sum $(\pi_D, H_D) = (\bigoplus_{x} \pi_x,\bigoplus_x H_x)$ is a faithful $\ast$-representation   of $C^*_D(\G)$ and its restriction to $L_1(\G)$ is a $D$-representation by Lemma \ref {lemproduct}.  
 
(\ref{ideals_inclusion}). If $D_1$ and $D_2$ are as above, then from the definition of the norms $\|\cdot\|_{D_i}$, we have 
\[\|\omega\|_{D_1} \le \|\omega\|_{D_2} \qquad (\omega \in L_1 (\G)), \] so the existence of the quotient map $\sigma:C_{D_2}^*(\G) \to C_{D_1}^*(\G)$ follows.  

(\ref{full}). Since $C_u(\hat \G)$ is the universal enveloping C$^\ast$-algebra of $L_1 (\G)$ and every representation  $\pi:L_1(\G) \to \mc B(H)$ is an $C_b(\G)$-representation, the canonical isomorphism $C_{\infty}^*(\G) \cong C_u(\hat \G)$  is immediate.

(\ref{unimodular}).  Let $\G$ be discrete and $C_c(\G) \subseteq D \subseteq  {\frak N}_{h_L} ^*$.   By Proposition \ref{prop:regular_is_C00}, the left regular representation $\lambda$ is a $C_c(\G)$-representation and therefore $\lambda$ extends 
to a surjective $\ast$-homomorphism \[\lambda:C^*_{C_c(\G)}(\G) \to C(\hat \G).\]  Combining this with (\ref{ideals_inclusion}), we obtain a sequence of surjective $\ast$-homomorphisms 
 \[
 C^*_{ {\frak N}_{h_L} ^*}(\G) \to C^*_{D}(\G) \to C^*_{C_c(\G)}(\G)  \xrightarrow{\lambda} C(\hat \G)
 \]
extending the identity map on $L_1(\G)$, and it suffices to show that these maps are isometric.  I.e.,  
 $\|\omega\|_{{\frak N}_{h_L}^*} \le \|\lambda(\omega)\|$ for each $\omega \in L_1^\sharp(\G)  $.  Fix $\omega \in L_1^\sharp(\G)  $ and $\epsilon >0$.  Then there is an ${\frak N}_{h_L} ^*$-representation $\pi:L_1(\G)\to \mc B(H)$ with dense subspace $H_0$ and a unit vector $\xi \in H_0$ so that \[\|\omega\|^2_{ {\frak N}_{h_L} ^*} - \epsilon \le \|\pi(\omega)\xi\|^2 = \langle  \omega^\sharp \star \omega, \varphi^\pi_{\xi,\xi}\rangle.\]
Since $\varphi^\pi_{\xi,\xi} \in {\frak N}_{h_L} ^*$, we can conclude from  Proposition  \ref {prop:C_00inA(G)} that $\varphi^\pi_{\xi,\xi}$ is a norm one positive definite function associated with $\lambda$. 
Therefore, we have 
\[
 \langle  \omega^\sharp \star \omega, \varphi^\pi_{\xi,\xi}\rangle \le \|\lambda(\omega^\sharp \star \omega)\| = \|\lambda(\omega)\|^2.\] 
Since $\epsilon > 0$ is arbitrary, we have $\|\omega\|_{{\frak N}_{h_L} ^*}   \le \|\lambda(\omega)\|$.
\end{proof}

\subsection{The coefficient spaces $B_D(\G)$ and $A_D(\G)$} \label{section:FSA}

Two important objects in our study of $D$-C$^\ast$-algebras will be the spaces $B_D(\G)$ and $A_D(\G)$ of coefficient functions associated with these C$^\ast$-algebras.  In this section, we define these objects and discuss a few of their basic properties.  For locally compact groups, spaces of this type were studied in \cite{KaLaQu}.  

As usual, let $D \subseteq C_b(\G)$ be a linear subspace for which $\|\cdot\|_D$ is non-degenerate on $L_1(\G)$.  Given a representation $(\pi,H_\pi)$ of $\G$, we say that  $\pi$ is \textit{$C^*_D(\G)$-continuous} if $\pi$ extends to a representation of $C^*_D(\G)$ on $H_\pi$.  We let 
\[
B_D(\G) = \Big\{ \varphi^{\pi}_{\xi,\eta} : 
(\pi, H_\pi) \text{ is a $C^*_D(\G)$-continuous representation of $\G$ and } \xi, \eta \in H_\pi \Big\}
\]
be the set of all coefficient functions of $C^*_D(\G)$-continuous representations of $\G$.
Since the direct sum of $C^*_D(\G)$-continuous representations is $C^*_D(\G)$-continuous, it is easy to see that $B_D(\G)$ is a linear subspace of $C_b(\G)$.
There is a natural Banach space norm on $B_D(\G)$ given by 
\[
\|\varphi\|_{B_D(\G)} = \inf\{\|\xi\|\|\eta\| : \varphi = \varphi^\pi_{\xi,\eta},  \ (\pi, H_\pi) \text{ is }\ C^*_D(\G)\mbox{-continuous and } \xi, \eta \in H_\pi\} .
\]  
With this norm, we can  isometrically identify $B_D(\G)$ with the dual Banach space $C^*_D(\G)^*$.
Indeed, if $\hat \mu \in C_D^*(\G)^*$, there exist a $C^*_D(\G)$-continuous representation $(\pi,H)$ of $\G$ and vectors $\xi, \eta \in H$ such that $\hat \mu = \omega_{\xi,\eta}\circ \pi$.   
Since $\|\hat\mu\| = \inf\{\|\xi\|\|\eta\|: \hat \mu = \omega_{\xi,\eta} \circ \pi\}$, the map 
\[
\hat \mu = \omega_{\xi, \eta} \circ \pi \in C^*_D(\G)^* \mapsto 
\varphi^{\pi} _{\xi, \eta}  \in B_D(\G) 
\] 
defines a linear isometric isomorphism from $C^*_D(\G)^*$ onto $B_D(\G)$.

\begin{rem}\label{rem:opposite}
When $D = C_b(\G)$, we know from Lemma \ref{lem:products_coefficients} that $B_{C_b(\G)}(\G)$ is a subalgebra of $C_b(\G)$, while on the other hand the dual space  
$C^*_{\infty}(\G)^* = C_u(\hat \G)^*$ is a completely contractive Banach algebra with product given by $\hat \mu \star \hat \nu = (\hat \mu \otimes \hat \nu)\circ\hat \Delta_u$. With respect to these algebraic structures, $B_{C_b(\G)}(\G)$ and $C^*_{\infty}(\G)^*$ are isometrically \textit{anti-isomorphic}.  Indeed, the linear isometric isomorphism 
  \[
  \hat \mu = \omega_{\xi, \eta} \circ \pi \in C^*_{\infty}(\G)^* \mapsto \varphi^{\pi}_{\xi, \eta} = (\iota \otimes  \omega_{\xi, \eta} \circ \hat \pi)
   \text{\reflectbox{$\W$}} \in B_{C_b(\G)}(\G)
   \]  
is anti-multiplicative since 
 $
(\iota \otimes \hat \Delta_u) \text{\reflectbox{$\W$}} = \text{\reflectbox{$\W$}}_{13} \text{\reflectbox{$\W$}}_{12}$. 
Following the existing conventions for locally compact groups, we write  $B(\G)= B_{C_b(\G)}(\G)$ and we call $B(\G)$ the \textit{Fourier-Stieltjes algebra} of $\G$.  In particular, since $C^*_{\infty}(\G)^*$ is a \textit{dual} Banach algebra \cite[Lemma 8.1]{Da}, it follows that $B(\G)$ is also a dual Banach algebra.  (I.e., the multiplication map on $B(\G)$ is separately weak$^*$ continuous).
\end{rem}
 
It follows from Assumption \ref{assumption} and Proposition \ref {equiv2} that if $D_1 \subseteq D_2 \subseteq C_b(\G)$ are linear subspaces, then there are surjective $\ast$-homomorphisms 
\[
C_u(\hat \G) = C^*_{\infty}(\G) \to C^*_{D_2}(\G) \to C^*_{D_1}(\G) \to C(\hat \G)
\]
extending the identity map on $L_1(\G)$.
Taking adjoints, we obtain weak$^\ast$-continuous isometric inclusions
\[
B_\lambda(\G) \hookrightarrow B_{D_1}(\G) \hookrightarrow B_{D_2}(\G) \hookrightarrow   B(\G),
\]
where $B_\lambda(\G) = \{\varphi^\pi_{\xi,\eta} \in B(\G): (\pi,H_\pi)\text{ is } C_0(\hat \G)\text{-continuous}\} \cong C_0(\hat \G)^*$ is the  \textit{reduced Fourier-Stieltjes} algebra of $\G$. 
We will show in Proposition \ref{prop:algebra} that if $D$ is any subalgebra (or a two-sided ideal) of $C_b(\G)$, then $B_D(\G)$ is a closed subalgebra (or a closed two-sided ideal) of $B(\G)$.

On the other hand, we let
\[
A_D(\G) = \Big\{ \varphi^{\pi}_{\xi,\eta}:  (\pi, H_\pi)   \text{ is a $D$-representation of $\G$} ~ \mbox{and} ~  \xi, \eta \in H_\pi \Big\} 
\]
be the set of all coefficient functions of $D$-representations of $\G$.
Since the direct sum of $D$-representations is again a $D$-representation (cf. Lemma \ref {lemproduct})
and every $D$-representation is $C^*_D(\G)$-continuous, 
$A_D(\G)$ is a linear subspace of $B_D(\G)$. There is a natural norm on $A_D(\G)$ given by 

\[\|\varphi\|_{A_D(\G)} = \inf\{\|\xi\|\|\eta\|: \ \varphi = \varphi^\pi_{\xi,\eta},  (\pi, H_\pi)   \text{ is a $D$-representation of $\G$} ~ \mbox{and} ~  \xi, \eta \in H_\pi\} . 
\] 
In general, we have  $\|\varphi\|_{A_D(\G)} \ge \|\varphi\|_{B_D(\G)} $ for all $\varphi\in A_D(\G)$. Since $A_D(\G)$ clearly separates points in $C^\ast_D(\G)$ (when viewed as a linear subspace of $C^*_D(\G)^*$ under our linear isomorphism $B_D(\G) \cong C_D^*(\G)^*$), we obtain a contractive weak$^*$ dense inclusion  $A_D(\G) \hookrightarrow B_D(\G)$.  

If  $D = C_b(\G)$, it is easy to see that  $A_{C_b(\G)}(\G) = B(\G)$ isometrically, but in general it is not clear whether the inclusion $A_D(\G) \hookrightarrow B_D(\G)$ is always isometric.
In the following proposition, we show that for a discrete quantum group $\G$ and any linear subspace $C_c(\G) \subseteq D \subseteq {\frak N}^*_{h_L}$, $A_D(\G)$ can be isometrically identified with the \emph{Fourier algebra} $A(\G)$ of $\G$, where $A(\G) = \{\varphi^\lambda_{\xi, \eta}: \xi, \eta \in L_2(\G)\}$ is the space of all coefficient functions of the left regular representation, equipped with the 
norm $\|\varphi\|_{A(\G)} = \inf \{\|\xi\| \|\eta\|: \varphi = \varphi^\lambda_{\xi, \eta}, \xi, \eta \in L_2(\G)\}$.
With this norm, $A(\G)$ is isometrically isomorphic to the predual of $L_\infty(\hat \G)$ and we have the isometric inclusions 
\[
 A(\G)  \hookrightarrow B_\lambda(\G) \hookrightarrow B(\G). 
\]
 
\begin{prop} \label{AG}
Let $\G$ be a discrete quantum group and let $C_c(\G) \subseteq D \subseteq L_\infty(\G)$ be a linear subspace .
\begin{itemize}
\item [(1)]  For any  $ \varphi  \in A(\G)$,  we have 
\[
\|\varphi\|_{A(\G)}  = \|\varphi\|_{A_{C_c(\G)}(\G)}  = \|\varphi\|_{A_D(\G)} = \|\varphi\|_{B(\G)} .\]
\item[(2)] If $D \subseteq {\frak N}^*_{h_L}$, then 
$A_D(\G)$ is  isometrically isomorphic to $A(\G) $.  In this case, we have the  isometric inclusion $A_{D}(\G) \hookrightarrow B_{D}(\G)$.
\end{itemize}
\end{prop}
\begin{proof}
Let us first recall from Proposition  \ref {prop:regular_is_C00}  that 
 the left regular representation $(\lambda, L_2(\G))$ is a $C_c(\G)$-representation.
Then we have 
\[
A(\G) \subseteq A_{C_c(\G)}(\G) \subseteq A_D(\G) \subseteq B(\G)
\]
 and 
\[
\|\varphi\|_{A(\G)}  \ge \|\varphi\|_{A_{C_c(\G)}(\G)}  \ge \|\varphi\|_{A_D(\G)} \ge \|\varphi\|_{B(\G)} \qquad (\varphi \in A(\G)).
\]
(1) now follows from the isometric inclusion 
$A(\G) \hookrightarrow B_\lambda(\G) \hookrightarrow B(\G)$.

To prove (2), it suffices to show that if  $C_c(\G) \subseteq D \subseteq {\frak N}^*_{h_L}$ and $\varphi \in A_D(\G)$, then $\varphi\in A(\G)$.
Now let us assume $\varphi = \varphi^\pi_{\xi, \eta}$ for some 
$D$-representation $(\pi, H)$ and   $\xi, \eta \in H$.
We can  choose two sequences of vectors $(\xi_n)$ and $(\eta_n)_n$ in the dense subspace $H_0\subseteq H$ such that 
$\xi_n \to \xi$ and $\eta_n \to \eta$. Then 
 $\varphi_n = \varphi^\pi_{\xi_n, \eta_n}\in D$ is a  sequence of $D$-coefficient functions such that 
 $\|\varphi_n - \varphi\|_{A_D(\G)} \to 0$.  Moreover, from Proposition \ref {prop:C_00inA(G)} we know that each $\varphi_n \in D$ is a coefficient function of $(\lambda, L_2(\G))$ and thus is contained in $A(\G)$.  Since $(\varphi_n)_n$ is a Cauchy sequence in $A_D(\G)$,
we conclude from (1) that $(\varphi_n)$ is also a Cauchy sequence in  $A(\G)$ (since $\|\varphi_n - \varphi_m\|_{A(\G)} = 
\|\varphi_n - \varphi_m\|_{A_D(\G)}$).
Therefore, there exists $\psi \in A(\G)$ such that $\|\varphi_n -\psi\|_{A(\G)} \to 0$. This implies that $\psi \in A_D(\G)$ and 
$\varphi_n\to \psi$ in $A_D(\G)$.  So we must have $\varphi = \psi \in A(\G)$.
\end{proof}

We conclude this subsection with the following result which shows that whenever the linear subspace $D \subseteq C_b(\G)$ under consideration is a subalgebra or a two-sided ideal, the
corresponding coefficient spaces $A_D(\G)$ and $B_D(\G)$ inherit these algebraic structures as linear subspaces of $B(\G)$. 

\begin{prop} \label{prop:algebra}
Let $\G$ be a locally compact quantum group and $D \subseteq C_b(\G)$ be a linear subspace (satisfying Assumption \ref{assumption}).
\begin{itemize}
\item [(1)]  If $D$ is a subalgebra of $C_b(\G)$, then  $A_D(\G)$ and $B_D(\G)$ are subalgebras of $B(\G)$.
\item [(2)]  If $D$ is a two-sided ideal in  $C_b(\G)$, then $A_D(\G)$ and $B_D(\G)$ are two-sided ideals in $B(\G)$.
\end{itemize}
In these cases, we call $A_D(\G)$ the \emph{$D$-Fourier algebra} of $\G$, and call $B_D(\G)$ the \emph{$D$-Fourier-Stieltjes algebra} of $\G$.
\end{prop}
\begin{proof} 
(1). It follows from Lemma \ref {lemproduct}  that if $D$ is a subalgebra of $C_b(\G)$, then the tensor product of $D$-representations  is again a $D$-representation.  We therefore conclude from Lemma \ref {lem:products_coefficients} that $A_D(\G)$ is a subalgebra of $B(\G)$.  Since $A_D(\G)$ is weak$^*$ dense in $B_D(\G)$ and the isometric inclusion $B_D(\G) \hookrightarrow B(\G)$ is weak$^\ast$-continuous, the fact that $B_D(\G)$ is a subalgebra of $B(\G)$ is just a routine calculation using the separate weak$^*$-continuity of multiplication in the dual Banach algebra $B(\G)$. 

(2). If $D$ is a two-sided ideal in $C_b(\G)$, then $A_D(\G)$ is a two-sided ideal in $B(\G)$ by Lemma \ref {lemproduct}. The fact that $B_D(\G)$ is a two-sided ideal
  in $B(\G)$ follows from the same reasoning as in (1). 
 \end{proof}

\subsection{$D$-C$^\ast$-algebras as compact quantum groups}

In this section (unless stated otherwise), $\G$ will be a discrete quantum group with compact dual quantum group $\hat \G$.  

Let $\|\cdot\|$ be a C$^\ast$-norm on the algebra of polynomial functions $\Pol(\hat \G)$ and denote by $\hat A$ the unital C$^\ast$-algebra obtained by completing $\Pol(\hat \G)$ with respect to this norm.  Following \cite{KySo}, we call $\|\cdot\|$ a \textit{quantum group norm} on $\Pol(\hat \G)$ if the coproduct $\hat \Delta_u: \Pol(\hat \G) \to \Pol(\hat \G) \odot \Pol(\hat \G)$ extends continuously to a $\ast$-homomorphism $\hat \Delta_{\hat A}: \hat A \to \hat A \otimes \hat A$.  In this case, the pair $(\hat A , \hat \Delta_{\hat A})$ then becomes a compact quantum group in the sense of Woronowicz \cite{Wo2}.  

The following proposition shows that $D$-C$^\ast$-algebra norms provide natural examples of quantum group norms whenever $C_c(\G) \subseteq D \subseteq L_\infty(\G)$ is a subalgebra.  

\begin{prop} \label{prop:cqg_structure}
Let $\G$ be a discrete quantum group and $C_c(\G) \subseteq D \subseteq L_\infty(\G)$ a subalgebra.  Then $\|\cdot\|_D$ is a quantum group norm on $\Pol(\hat \G)$. 
\end{prop}

\begin{proof}
Let $\hat \pi_D:C_D(\hat \G) \to \mc B(H_D)$ be a faithful $\ast$-representation such that the corresponding map $\pi_D:L_1(\G) \to \mc B(H_D)$ is a $D$-representation with canonical dense subspace $H_0 \subset H_D$.
 We need to show that the  $\ast$-homomorphism 
\[\hat \Delta_D:= (\hat \pi_D \otimes  \hat \pi_D) \circ \hat \Delta_u: \Pol(\hat \G) \to \mc B(H_D \otimes H_D)\] is continuous for the $\|\cdot\|_D$-norm on $\Pol(\hat\G)$.  
For this, it suffices to show that the representation $\rho_D:L_1(\G) \to \mc B(H_D \otimes H_D)$ corresponding to $\hat \Delta_D$ is a $D$-representation.  To this end, let $U \in M(C_0(\G) \otimes \mc K(H_D))$ and $V \in M(C_0(\G) \otimes \mc K(H_D \otimes H_D)) $ be the unitary representations of $\G$ corresponding to  $\pi_D$ and $\rho_D$, respectively.  Since $(\iota \otimes \hat \Delta_u)(\text{\reflectbox{$\W$}}) = \text{\reflectbox{$\W$}}_{13}\text{\reflectbox{$\W$}}_{12}$, we have 
\[V = (\iota \otimes \hat \Delta_D)\text{\reflectbox{$\W$}} = (\iota \otimes \hat \pi_D \otimes \hat \pi_D)\text{\reflectbox{$\W$}}_{13}\text{\reflectbox{$\W$}}_{12} = U_{13}U_{12}. 
\]  
In particular, for all $\xi_1,\xi_2, \eta_1, \eta_2 \in H_0$ we have
\begin{align} \label{eqn:multduality}\varphi^{\rho_D}_{\xi_1\otimes \xi_2, \eta_1\otimes \eta_2} = (\iota \otimes \omega_{\xi_1,\eta_1} \otimes \omega_{\xi_2,\eta_2})U_{13}U_{12} = \varphi^{\pi_D}_{\xi_2,\eta_2}\varphi^{\pi_D}_{\xi_1,\eta_1} \in D.
\end{align}  Therefore $\rho_D$ is a $D$-representation with canonical dense subspace given by $H_0 \odot H_0$.       
\end{proof}

\begin{rem} \label{rem:coactions}
When $C_c(\G) \subseteq D_1 \subseteq D_2 \subseteq L_\infty(\G)$ are two-sided ideals, then the above proof can be easily adapted to show that $\hat \Delta_u:\Pol(\hat \G) \to \Pol(\hat \G)  \odot \Pol(\hat \G) $ extends to a unital $\ast$-homomorphism 
\[
\hat \delta:C_{D_1}(\hat \G) \to C_{D_2}(\hat \G)  \otimes C_{D_1}(\hat \G)
\]
 such that $(\iota \otimes \hat \delta) \circ \hat \delta = (\hat \Delta_{D_2} \otimes \iota)\hat \delta$.  That is, $\hat \delta$ defines a   \textit{left action} of the compact quantum group $ (C_{D_2}(\hat \G), \hat \Delta_{D_2})$ on the unital C$^\ast$-algebra $C_{D_1}(\hat \G)$. 
To prove this result, one just needs to replace $\pi_D$ in the above argument by appropriate $D_i$-representations $\pi_{D_i}$ such that the corresponding representations $\hat \pi_{D_i}:C_{D_i}(\hat \G) \to \mc B(H_{D_i})$ are faithful. Then $\hat \delta|_{\Pol(\hat \G)} = (\hat \pi_{D_2} \otimes \hat \pi_{D_1}) \circ \hat \Delta_u$.  
 We leave the details to the reader.
\end{rem}

Using Proposition \ref{prop:cqg_structure}, we can obtain a refinement of Proposition \ref{prop:algebra} by concluding that the $D$-Fourier-Stieltjes algebras are completely contractive Banach algebras.

\begin{cor}\label{cor:FS_BA} 
For each subalgebra  $C_c(\G) \subseteq D \subseteq L_\infty(\G)$, $B_D(\G)$ is a completely contractive Banach algebra (with respect to its canonical operator space structure as the dual space of $C_D(\hat \G)$).
\end{cor}

\begin{proof}
Let $\varphi, \psi \in B_D(\G)$ and let $\hat\mu, \hat\nu \in C_D(\hat \G)^*$ be the corresponding linear functionals under the completely isometric identification $B_D(\G) \cong C_D(\hat \G)^*$. 
In this case, we can express $\varphi = (\iota \otimes \hat\mu \circ \hat \pi)\text{\reflectbox{$\W$}}$ and $\psi = (\iota \otimes \hat\nu \circ \hat \pi)\text{\reflectbox{$\W$}}$, 
 where $\hat \pi:C_u(\hat \G) \to  C_D(\hat \G)$ is the canonical quotient map.  Now consider the product $\varphi \psi$, which belongs to $B_D(\G)$ by Proposition \ref{prop:algebra}.  Since the coproduct $\hat \Delta_D$ satisfies $\hat \Delta_D \circ \hat \pi = (\hat \pi \otimes \hat \pi) \circ \hat \Delta_u$, we obtain
\begin{align*}\varphi \psi &= (\iota \otimes (\hat \mu \circ \hat \pi) \otimes (\hat \nu
\circ \hat \pi)) (\text{\reflectbox{$\W$}}_{12}\text{\reflectbox{$\W$}}_{13}) =(\iota \otimes (\hat\mu \circ \hat \pi) \otimes (\hat\nu
\circ \hat \pi)) (\iota \otimes \sigma) (\text{\reflectbox{$\W$}}_{13}\text{\reflectbox{$\W$}}_{12}) \\
&= (\iota \otimes( \hat\mu \otimes \hat \nu)\circ(\sigma (\hat \pi \otimes \hat \pi)\hat \Delta_u))(\text{\reflectbox{$\W$}})= (\iota \otimes( \hat\mu \otimes \hat\nu)\circ(\sigma \hat \Delta_D \circ \hat \pi))(\text{\reflectbox{$\W$}}),
\end{align*}
 where $\sigma$ denotes the tensor flip automorphism.  This shows that under the identification $B_D(\G) \cong C_D(\hat \G)^*$, the multiplication map $m:B_D(\G) \odot B_D(\G) \to B_D(\G)$, which corresponds to the map $(\sigma \hat \Delta_D)^*:C_D(\hat\G)^*\widehat{\otimes }C_D(\hat\G)^* \to C_D(\hat\G)^*$, is completely contractive since $\sigma$ and $\hat \Delta_D$ are both $\ast$-homomorphisms.  Therefore $B_D(\G)$ is a completely contractive Banach algebra.
\end{proof}

\begin{rem} \label{rem:lccase}
Certain results from this section admit partial generalizations to the locally compact case.  Let  $\G$ be a (not necessarily discrete) locally compact quantum group and $D \subseteq C_b(\G)$ a subalgebra satisfying Assumption \ref{assumption}.  Then $B_D(\G)$ is a subalgebra of $B(\G)$ by Proposition \ref{prop:algebra}. Since $B(\G)$ is completely isometrically anti-isomorphic to $C_u(\hat \G)^*$ (with its canonical product $(\hat \Delta_u)^*$), we can use \cite[Lemma 3.15]{KaLaQu} to obtain a co-associative comultiplication \[\hat \Delta_D:C^*_D(\G) \to M(C^*_D(\G) \otimes C^*_D(\G)) \quad \text{such that} \quad  \hat \Delta_D \circ \hat \pi = (\hat \pi \otimes \hat \pi)\circ \hat \Delta_u,\]
where $\hat \pi:C_u(\hat \G) \to C^*_D(\G)$ is the canonical quotient map.  
We note, however, that when $\G$ is not discrete or compact it is not clear whether the pair $(C^*_D(\G), \hat \Delta_D)$ can be made into a genuine locally compact quantum group since we do not know how to construct an antipode in general.  
\end{rem}

\subsection{The Haagerup property and $C_0(\G)$-representations} \label{section:HAP}

In \cite{DaFiSkWh}, the definition of the Haagerup property for locally compact groups was extended to the setting of locally compact quantum groups.  In particular, one of the (several equivalent)  definitions of the \textit{Haagerup property} for a locally compact quantum group $\G$ is the existence of an approximate identity $(\varphi_i)_{i \in I}$ for $C_0(\G)$ consisting of norm one positive definite functions.  The following theorem was established for discrete groups by Brown and Guentner \cite{BrGu12}.  In the quantum case, the proof is similar.

\begin{thm} \label{thm:HAP}
For a locally compact quantum group $\G$, the following conditions are equivalent.
\begin{enumerate}
\item \label{HAP} $\G$ has the Haagerup property.
\item \label{C*equality} The canonical quotient map $C_u(\hat \G) \to C^*_{C_0(\G)}(\G)$ is an isomorphism.
\end{enumerate}
\end{thm}

\begin{proof}
\eqref{HAP} $\implies$ \eqref{C*equality}.  Let $(\varphi_i)_{i \in I}$ be an approximate identity for $C_0(\G)$ consisting of norm one positive definite functions, let $\varphi \in B(\G)$ be another norm one positive definite function and consider the net of norm one positive definite functions $(\varphi_i\varphi)_{i \in I} \subseteq C_0(\G)$.  Let $(\pi_i, H_i, \xi_i)$ be  GNS constructions for $\varphi_i\varphi$.  Then $\pi_i$ is a $C_0(\G)$-representation \cite[Lemma 3.3]{DaFiSkWh}.  Since we also have $\langle \omega^\sharp\star \omega, \varphi\rangle =  \lim_{i \in I} \langle \omega^\sharp\star \omega, \varphi_i\varphi\rangle$ for each $\omega \in L_1^\sharp(\G)$, it follows that 
\[
|\langle \omega^\sharp\star \omega, \varphi\rangle| =  \lim_{i \in I}| \langle \omega^\sharp\star \omega, \varphi_i\varphi\rangle| = \lim_{i \in I} \|\pi_i(\omega)\xi_i\|^2 \le \|\omega\|_{C_0(\G)}^2 \qquad (\omega \in L_1^\sharp(\G)).
\]  Taking the supremum over all norm one positive definite  functions $\varphi \in B(\G)$, we obtain 
$\|\omega\|_{C_u(\hat \G)}^2 \le \|\omega\|_{C_0(\G)}^2$ for each $\omega \in L_1^\sharp(\G)$, and therefore the quotient map $C_u(\hat \G) \to C^*_{C_0(\G)}(\G)$ is isometric.

\eqref{C*equality} $\implies$ \eqref{HAP}.  Denote by $\mc R$ the collection of all $C_0(\G)$-representations of $\G$ and let $\hat \epsilon:C_u(\hat \G) \to \C$ be the dual co-unit.  I.e,   the $\ast$-character given by $\hat \epsilon(\pi_u(\omega)) = \langle \omega, 1 \rangle$ for each $\omega \in L_1(\G)$.  Since $C_u(\hat \G) \cong C^*_{C_0(\G)}(\G)$, the representation $\hat \epsilon$ is weakly contained in $\mc R$, which by the remark following \cite[Theorem 2.1]{DaFiSkWh} means that there exists a net of cyclic $C_0(\G)$-representations $(\pi_i, H_i, \xi_i)_{i \in I} \subset \mc R$ such that $\hat \epsilon = \text{weak}^\ast - \lim_{i \in I} \omega_{\xi_i,\xi_i} \circ \hat \pi_i$.  For each $i \in I$, let $\varphi_i = \varphi^\pi_{\xi_i,\xi_i}$.  Since $\|\xi_i\|^2 = \langle \hat \pi_i(1_{C_u(\hat \G)}) \xi_i|\xi_i\rangle \to \hat \epsilon(1_{C_u(\hat \G)}) = 1$, we may assume without loss of generality that $\|\xi_i\| = 1$ for each $i$, and therefore $(\varphi_i)_{i \in I}$ is a net of norm one positive definite functions in $C_0(\G)$.  We now follow the argument in the proof of $(i) \implies (iii)$ in \cite[Theorem 5.5]{DaFiSkWh} to see that for each $x \in C_0(\G)$, $\varphi_i x \to x$ and $x \varphi_i \to x$ weakly in $C_0(\G)$.  By passing to convex combinations of the net $(\varphi_i)_{i \in I}$, a Hahn-Banach argument yields a new net of norm one positive definite functions $(\psi_j)_{j \in J} \subset C_0(\G)$ that is an approximate identity for $C_0(\G)$.
\end{proof}


\section{Unimodular discrete quantum groups and the ideals $L_p(\G)$} \label{section:unimodular}
 
Let us recall that a discrete quantum group $\G$ is    \textit{unimodular} (or a discrete \emph{Kac algebra}) if $Q_\alpha = 1$ for all 
$\alpha \in \text{Irr}(\hat \G)$. 
In this case, the antipode $S$ is an isometric anti-automorphism on 
$L_\infty(\G)$ and thus $\sharp$ defines a completely isometric involution on $L_1(\G) = L_1^\sharp (\G)$.  
We have $m_\alpha = d_\alpha$ and 
$h=  h_L =  h_R$  is a semifinite trace on $L_\infty(\G) $, which is given by 
\[
h =\bigoplus_{\alpha \in \Irr(\hat \G)}  d_\alpha \Tr_{\alpha}.
\]
Unimodularity of $\G$ is also equivalent to the traciality of the Haar state $\hat h$ on the dual quantum group $\hat \G$.   

For  $1 \le p < \infty$, the induced non-commutative $L_p$-space $L_p(\G) = L_p(L_\infty(\G), h)$ is just the $\ell_p$-direct sum of Shatten $p$-classes
\[
L_p(\G) = \ell_p - \bigoplus_{\alpha \in \Irr(\hat \G)}  L_p(\mc B(H_\alpha), d_\alpha \Tr_{\alpha}).
\]
Here we identify $L_1(\G) $ with ${\frak M}_h$ and $L_2(\G) $ with ${\frak N}_h$ when we regard them as subspaces of 
$L_\infty(\G)$.
The norm $\|\cdot\|_p$ on $L_p(\G)$  is given by 
\[
\|x\|_{p} =  h(|x|^p)^{1/p}   =    ( \sum_{\alpha \in \Irr(\hat \G)} d_\alpha \Tr_{\alpha}(|p_\alpha x |^p)   )^{\frac 1p}
    \qquad (x \in L_p(\G)),
\]     
and $L_p(\G)$ is an $L_\infty(\G)$-bimodule satisfying
\[
\|x y x\|_p \le \|x\|_\infty \|y\|_p\|z\| _\infty
\quad ( x, z \in L_\infty(\G), y\in L_p(\G)).
\]
In particular, the spaces $L_p(\G)$ give rise to examples of two-sided ideals in $L_\infty(\G)$.  

The following lemma is easy.

\begin{lem} \label{lem:Lpideal}
If $|\Irr(\hat \G)| = \infty$ and $1 \le p_1 < p_2 \le \infty$, then $L_{p_1}(\G)$ is a proper subspace of $L_{p_2}(\G)$ with contractive inclusions
\[
L_1(\G) \hookrightarrow L_{p_1}(\G)\hookrightarrow L_{p_2}(\G) \hookrightarrow  C_0(\G) \subseteq L_\infty(\G).
\]
\end{lem}    
When dealing with elements $x \in {\frak M}_h$ we will at times write $\|\omega_x\|_p$ instead of $\|x\|_p$ when we wish to emphasize the identification ${\frak M}_h \cong L_1(\G)$. 

\begin{rem} 
We can also obtain the non-commutative $L_p$-space $L_p(\G)$ as
a natural  complex interpolation space.  See \cite[Chapter 7]{Pi}.
For $1 \le p_1 < p < p_2 \le \infty$, there exists a positive number  $\theta \in (0, 1)$ such that 
${\frac  1p} = {\frac  \theta {p_1}} + {\frac {1-\theta} {p_2}}$.
In this case, we have 
\[
L_p(\G) = (L_{p_2}(\G), L_{p_1}(\G))_{\theta}.
\]
If $T$ is a  linear map, which is defined and bounded on both $L_{p_1}(\G)$ and $L_{p_2}(\G)$, then $T$ is bounded on $L_{p}(\G)$
with 
\[
\|T\|_{L_{p}(\G)} \le \|T\|_{L_{p_1}(\G)}^\theta  \|T\|_{L_{p_2}(\G)} ^{1- \theta}.
\]
This fact will be used in the next section.
\end{rem}

\subsection{The left and right actions of $L_1(\G)$ on $L_p(\G)$}

For a unimodular discrete quantum group $\G$ and $p \in [1,\infty]$,    $L_1(\G)$ acts on $L_p(\G)$ in a natural way that extends the left  action $\alpha_1$ of $L_1(\G)$ on itself by convolution: 
\[
\alpha_1(\omega)(\omega' )= \omega \star \omega' \qquad (\omega,\omega' \in L_1(\G)).
\]
Let us recall how this is done.

Motivated by the group case, we can similarly define a contractive left action $\alpha_\infty$  of $L_1(\G)$ on  $L_\infty(\G)$ given by 
\[
\alpha_\infty(\omega)(x) = (\omega \circ S \otimes \iota) \Delta(x) = x \star S_*(\omega) \qquad (\omega \in L_1(\G), \ x \in L_\infty(\G)). 
\]

The following lemma was observed by Kalantar \cite{Ka13}, which shows that the above actions on $L_1(\G)$ and $L_\infty(\G)$ are compatible under the identification $L_1(\G) = {\frak M}_{h} \subseteq L_\infty(\G)$. 

\begin{lem} \label{lem:compatibility}
Let $x \in {\frak M}_h$ and $\omega\in L_1(\G)$.  
Then $\alpha_\infty(\omega)(x) \in \mf M_h$ and \[\omega_{\alpha_\infty(\omega)(x)}=\omega\star \omega_x = \alpha_1(\omega)(\omega_x) .\]
In particular, $\alpha_\infty(\omega_y)(x) = y\star x \in C_c(\G)$ for each $x,y \in C_c(\G)$.
\end{lem}
\begin{proof}
Let us first recall  that for the (left=right) Haar weight $h$, we have
from \cite[Proposition 5.20]{KuVa1} that
\[
(\iota \otimes h) \Delta (y) (1 \otimes x) = S( (\iota \otimes h) (1 \otimes y) \Delta(x)) \qquad (x,y \in \mf N_{h}).
\]
Since $\alpha_\infty(\omega)(x) =  (\omega \circ S \otimes \iota) \Delta(x)$,  for any $y\in \mf N_{h} \subset L_\infty(\G)$ we have 
\begin{eqnarray*}
\langle \omega_{\alpha_\infty(\omega)(x)}, y\rangle &=&
h (y \, \alpha_\infty(\omega)(x)) 
= \langle \omega \circ S,  (\iota \otimes h) (1 \otimes y)\Delta (x)  \rangle\\
&=& \langle \omega \otimes h,   \Delta (y) (1 \otimes x)\rangle 
= \langle \omega \otimes\omega_x,   \Delta(y) \rangle = \langle \omega \star\omega_x, y \rangle\\
&=& \langle \alpha_1(\omega)(\omega_x), y \rangle .
\end{eqnarray*}
Since $\mf N_h$ is $\sigma$-weakly dense in $L_\infty(\G)$, this proves the second assertion, i.e. 
 $\omega_{\alpha_\infty(\omega)(x)} = \alpha_1(\omega)(\omega_x)$ for each $x \in \mf M_h$ and $\omega \in L_1(\G)$.
 The first and third assertions follow immediately from this calculation and the definition of the convolution product on $C_c(\G)$. 
\end{proof}

Therefore, for each $\omega \in L_1(\G)$, the induced action map 
$
\alpha_\infty(\omega): x \in L_\infty(\G) \mapsto x \star S_*(\omega) \in L_\infty(\G)
$
 and the induced action map $\alpha_1(\omega): \omega_x \in L_1(\G) \mapsto \omega \star \omega_x  \in L_1(\G)$ are compatible maps with norms $\le \|\omega\|_1$.
So for $1 < p < \infty$,
we can take the complex interpolation to obtain a map $\alpha_p(\omega) \in \mc B(L_p(\G))$ with norm $\le \|\omega\|_1$.  
Since $\G$ is unimodular with  tracial Haar weight $h = h_L = h_R$ on $\G$, 
we can similarly define the right $L_1(\G)$ action $\beta_p$ on $L_p(\G)$ 
by  complex interpolation.  More precisely, given any $\omega\in L_1(\G)$,
we can obtain  a map $\beta_p(\omega) \in \mc B(L_p(\G))$ with norm $\le \|\omega\|_1$ by  interpolating between the right actions $\beta_\infty:L_1(\G) \to \mc B(L_\infty(\G))$; $\beta_\infty(\omega)(x) = (\iota \otimes \omega \circ  S)\Delta(x) = S_*(\omega) \star x$ 
and $\beta_1:L_1(\G) \to \mc B(L_1(\G))$; $\beta_1(\omega)(\omega_x )= \omega_x \star \omega$ for all $x \in 
{\frak M}_h$.

We call the maps $\alpha_p:L_1(\G) \to \mc B(L_p(\G))$ and $\beta_p:L_1(\G) \to \mc B(L_p(\G))$ the \textit{canonical left and right actions} of $L_1(\G)$ on $L_p(\G)$, respectively.  Observe that $\alpha_p$ and $ \beta_p$ are indeed left (resp. right) actions in the usual sense.

\subsection{$L_p$-representations and $L_p$-C$^\ast$-algebras}

In this section we prove a few basic properties of $L_p(\G)$-(or simply $L_p$-)representations of unimodular discrete quantum groups and their associated $L_p$-C$^\ast$-algebras $C^*_{L_p(\G)}(\G)$.  For notational simplicity, we will write $C^*_p(\G)$ instead of $C^*_{L_p(\G)}(\G)$ for the remainder of the paper.       

In \cite{BrGu12}, Brown and Guentner showed that for a discrete group $G$ and $1 \le p \le 2$, the $\ell_p$-C$^\ast$-algebra of $G$ is canonically isomorphic to the reduced C$^\ast$-algebra of $G$.  The generalization of this result to unimodular discrete quantum groups is an immediate consequence of Proposition \ref{equiv2} \eqref{unimodular}.  

\begin{prop}
Let $\G$ be a unimodular discrete quantum group.  For each $1\le p\le2$, the identity map on $L_1(\G)$ extends to a $\ast$-isomorphism $C^*_{p}(\G) \cong C(\hat \G)$.
\end{prop}

On the other hand,  Brown and Guentner \cite[Proposition 2.12]{BrGu12} showed that injectivity of the natural quotient map $C^*(G) \to C^*_{\ell_p}(G)$ for \textit{any} $1 \le p < \infty$ characterizes the amenability of $G$.  In general, we have the following natural quantum group analogue of this result.

\begin{thm} \label{thm:amenability}
Let $\G$ be a unimodular discrete quantum group. Then the following are equivalent.
\begin{enumerate}
\item \label {amenable1} $\G$ is amenable.
\item \label {Lp} the canonical quotient map  $C_{\infty}^*(\G) \to  C^*_{p}(\G)$ is an isomorphism for some (and thus for all)  $1\le p  <  \infty$.
\end{enumerate}
\end{thm} 
\begin{proof} $(1) \implies (2)$. For any $1\le p < \infty$, we have the canonical quotient maps 
\[
C_{\infty}^*(\G)\to C_{p}^*(\G) \to C_{C_c(\G)}^*(\G).
\]
For discrete quantum groups, it is known from  \cite{Ru96}
 and \cite{To} that $\G$ is amenable if and only if $\hat \G$ 
is co-amenable, i.e.  the canonical quotient map 
\[
C_{\infty}^*(\G)= C_u(\hat \G) \to  C^*_{C_c(\G)}(\G) = C(\hat \G)
\]
 is an isomorphism.   Therefore the canonical quotient map  $C_{\infty}^*(\G) \to  C^*_{p}(\G)$ is an
isomorphism.

$(2) \implies (1)$. Let us assume that the canonical quotient map $C_{\infty}^*(\G) \to  C^*_{p}(\G)$ is an isomorphism for some $1 \le p < \infty$.
Then we can find a positive integer $k \ge p$ and obtain the  isomorphisms
\[
C_{\infty}^*(\G)  =  C_{k}^*(\G)= C_{p}^*(\G).
\]
Let $\pi:  L_1(\G) \to \mc B(H)$ be a faithful $L_k(\G)$-representation such that $\pi$ extends to a $\ast$-isomorphism
$C_{k}^*(\G)  = \overline{\pi(L_1(\G))}^{\|\cdot\|} \subseteq \mc B(H)$.  By taking an amplification of $\pi$ if necessary, we may assume that $C_{k}^*(\G) \cap \mc K(H) = \{0\}$.
Let $\hat \epsilon$ be the co-unit of $C_{\infty}^*(\G)= C_u(\hat \G) = C_{k}^*(\G)$ given by $\hat \epsilon(\pi_u(\omega)) = \omega (1)$ for all $\omega\in L_1(\G)$.
To prove (2) implies (1),  it suffices to show that $\ker \hat \pi_\lambda \subseteq \ker \hat \epsilon $
 (see \cite[Theorem 3.1]{BeTu}).

To see this, we apply Glimm's Lemma \cite[Lemma II.5.1]{Da96}, which says that 
for the state $\hat \epsilon:C_{k}^*(\G) \to \C$ 
there exists a net of unit  vectors $(\xi_i)_{i\in I} \subset H$ such that 
\[
\hat \epsilon(\hat x) = \lim_{i}  \langle \hat x \xi_i |\xi_i\rangle  \qquad (\hat x \in C_{k}^*(\G)= C_{\infty}^*(\G) ).
\]
In particular, for $\omega \in L_1(\G)$, we get  
\begin{align}\label{eqn:weak}
\langle \omega, 1 \rangle  = \hat \epsilon(\pi(\omega)) = \lim_{i} \langle \pi(\omega) \xi_i |\xi_i\rangle =
\lim_{i} \langle \omega, \varphi^\pi_{\xi_i, \xi_i} \rangle.
\end{align}
That is, $\varphi^\pi_{\xi_i, \xi_i} \to 1$ $\sigma$-weakly in $L_\infty(\G)$.
Since $(\pi, H)$ is an $L_k(\G)$-representation, we may assume that each $\xi_i$ is chosen from the canonical dense subspace $H_0 \subset H$ and therefore $(\varphi^\pi_{\xi_i, \xi_i})_i$ is a net of norm one positive definite functions  in $L_k(\G)$.
In this case, H\"older's inequality implies that the $k$-fold product  
\[
(\varphi^\pi_{\xi_i, \xi_i})^k 
= \varphi^{\pi\otop \cdots \otop \pi}_{\xi_i \otimes \cdots \otimes \xi_i, \xi_i \otimes \cdots \otimes \xi_i}
\]
is a norm one positive definite function   in $L_1(\G) \subseteq L_2(\G)$.  Applying Proposition \ref {prop:C_00inA(G)} to $(\varphi^\pi_{\xi_i, \xi_i})^k$, we obtain a net of  unit vectors $(\eta_i)_{i \in I}  \subset  L_2(\G)$ such that 
\[
\varphi^\lambda_{\eta_i, \eta_i} =  (\varphi^\pi_{\xi_i, \xi_i})^k \in A(\G) \qquad (i \in I).
\]
To conclude the proof, it suffices to show that  $\varphi^\lambda_{\eta_i, \eta_i} \to 1$ in the weak* topology of $B(\G)$.  Indeed, if this is the case, then for any $\hat x \in \ker \hat \pi_\lambda$, we have $\hat \epsilon (\hat x) = \lim_i \varphi^\lambda_{\eta_i,\eta_i}(\hat x) = \lim_{i} \langle \hat \pi_\lambda(\hat x) \eta_i|\eta_i \rangle  = 0$, giving $\ker \hat \pi_\lambda \subseteq \ker \hat \epsilon$.

To establish that $\varphi^\lambda_{\eta_i, \eta_i} \to 1$ in the weak* topology of $B(\G)$, note that $(\varphi^\lambda_{\eta_i, \eta_i})_{i \in I} \in B(\G) \cong C^*_{k}(\G)^*$ is a bounded net and that $\pi(\mc A (\G))$ is norm dense in $C^*_{k}(\G)$. Therefore by linearity and density it suffices to check that 
\[\langle \omega_x, 1 \rangle = \lim_i\langle \omega_x, \varphi^\lambda_{\eta_i, \eta_i} \rangle \qquad (x \in \mc B(H_\alpha), \ \alpha \in \Irr(\hat \G)).\]  
But since $\varphi^\pi_{\xi_i, \xi_i} \to 1$ $\sigma$-weakly in $L_\infty(\G)$, for each $\alpha \in \Irr(\hat \G)$ we have  $\lim_i p_\alpha\varphi^\pi_{\xi_i, \xi_i} = p_\alpha \in \mc B(H_\alpha)$ in norm (because $\mc B(H_\alpha)$ is finite dimensional) and hence also \[p_\alpha \varphi^\lambda_{\eta_i, \eta_i} = (p_\alpha\varphi^\pi_{\xi_i, \xi_i})^k \to (p_\alpha)^k = p_\alpha \quad \text{in norm}.\]  In particular,  for each $x \in \mc B(H_\alpha)$, we get \[\langle \omega_x, 1 \rangle = \langle \omega_x, p_\alpha \rangle =  \lim_i\langle \omega_x, \varphi^\lambda_{\eta_i, \eta_i}p_\alpha \rangle =  \lim_i\langle \omega_x, \varphi^\lambda_{\eta_i, \eta_i} \rangle.\]
\end{proof} 

We end this section with a technical result which, among other things,  gives a sufficient condition for a cyclic representation to be an $L_p$-representation.  The discrete group analogue of this result was proved by Okayasu \cite{Ok12}.  This result will be used in Section \ref{section:applications}.

\begin{lem} \label{lem:Kac_result}
Let $\G$ be a unimodular discrete quantum group and $\pi:L_1(\G) \to \mc B(H)$ a cyclic representation with cyclic vector $\xi \in H$.  If the positive definite function  $\varphi^\pi_{\xi,\xi} = (\iota \otimes \omega_{\xi,\xi})U_\pi$ belongs to $L_p(\G)$ for some $1 \le p < \infty$, then $\pi$ is an $L_p$-representation.  Moreover, if $q$ is the conjugate exponent of $p$, we have 
\[\|\pi(\omega)\| \le \liminf_{k \to \infty} \|(\omega^\sharp \star \omega)^{\star 2k}\|_q^{\frac{1}{4k}} \qquad (\omega \in L_1(\G))= L_1^\sharp(\G)).\]
\end{lem}
\begin{proof}
Let $(\pi, H ,\xi)$ be as above and assume that $\varphi^\pi_{\xi,\xi} \in L_p(\G)$ for some $1 \le p < \infty$.  Then  $H_0 = \{\pi(\omega)\xi: \omega \in  L_1(\G)\}$ is a dense linear subspace of $H$.  For any  $a = \pi(\omega_1)\xi, b = \pi(\omega_2)\xi \in H_0$, we have from Lemma \ref{lem:leftright_action} and the definitions of the left and right actions $\alpha_p,\beta_p:L_1(\G) \to \mc B(L_p(\G))$ that
\[
\varphi^\pi_{a,b} = \omega_1 \star \varphi^\pi_{\xi,\xi}\star \omega_2^\sharp = \beta_p(S_*(\omega_1))\alpha_p(\omega_2^*)\varphi^\pi_{\xi,\xi} \in L_p(\G).
\]
Therefore $\pi$ is an $L_p$-representation.

To prove the second assertion, we use the following general fact about bounded operators on $H$ (see for example the proof of  \cite[Theorem 1]{CoHaHo88}).  Given $A \in \mc B(H)$ and any dense subspace $H_0 \subset H$, we have
\begin{align} \label{eqn:Haagerup}
\|A\| = \sup_{\eta \in   H_0} \lim_{k \to \infty} \langle (A^*A)^{2k} \eta |\eta\rangle^{\frac{1}{4k}}.
\end{align} 
Now let $H_0$ be the dense subspace defined above and $A = \pi(\omega)$ for some   $\omega \in L_1(\G)$.  Then by \eqref{eqn:Haagerup},
\[\|\pi(\omega)\| = \sup_{\eta \in   H_0} \lim_{k \to \infty} \langle (\pi(\omega^\sharp \star \omega))^{2k} \eta |\eta\rangle^{\frac{1}{4k}} =  \sup_{\eta \in   H_0} \lim_{k \to \infty}\langle (\omega_x^\sharp \star \omega_x)^{2k}, \varphi^\pi_{\eta,\eta}\rangle^{\frac{1}{4k}}. \]  But since  $\varphi^\pi_{\eta,\eta} \in L_p(\G)$ for all $\eta \in H_0$,  \[\langle (\omega_x^\sharp \star \omega_x)^{2k}, \varphi^\pi_{\eta,\eta}\rangle^{\frac{1}{4k}} \le \| (\omega^\sharp \star \omega)^{2k}\|_q^{\frac{1}{4k}}\| \varphi^\pi_{\eta,\eta}\|_p^{\frac{1}{4k}}.\]
Letting $k \to \infty$, we finally obtain $\|\pi(\omega)\| \le \liminf_{k \to \infty} \|(\omega^\sharp \star \omega)^{\star 2k}\|_q^{\frac{1}{4k}}$.
\end{proof}

\section{Applications: exotic quantum group norms for free quantum groups} \label{section:applications}

In this section, we use the theory developed in the previous sections to study the $L_p$-C$^\ast$-algebras of some interesting examples of unimodular discrete quantum groups.  We will mainly focus on the orthogonal free quantum groups.  In Section \ref{section:unitary} we
also show how some of the techniques and results in the orthogonal case can be extended to  the case of the unitary free quantum groups. 
\begin{defn}[\cite{VaWa}]
Let $N \ge 2$ be an integer, let $F \in \text{GL}_N(\C)$ be such that $F \bar F = \pm 1$.  The \textit{orthogonal free quantum group} (\textit{with parameter matrix $F$}) is the discrete quantum group $\F O_F$ whose compact dual quantum group $O^+_F = \widehat{\F O_F}$ is given by the universal C$^\ast$-algebra 
\[C_u(O^+_F) = C^*\big(\{\hat u_{ij}\}_{1 \le i,j \le N} \ | \ \hat U = [\hat u_{ij}] \text{ is unitary and }  \hat U = F \bar{ \hat U} F^{-1}\big),\] where $\bar{\hat U} = [\hat u_{ij}^*]$. The coproduct $\hat \Delta_u:C_u(O^+_F) \to C_u(O^+_F) \otimes C_u(O^+_F)$ is determined uniquely by the equations
\[\hat \Delta_u(\hat u_{ij}) = \sum_{k=1}^N \hat u_{ik} \otimes \hat u_{kj} \qquad (1 \le i,j \le N).\]  
\end{defn}

Let us first recall a few important facts about the quantum groups $\F O_F$. 

\begin{rems}
\begin{enumerate}
\item Note that the coproduct $\hat \Delta_u$ is defined so that $\hat U$ is always a unitary representation of $O^+_F$.  $\hat U$ is called the \textit{fundamental representation} of $O^+_F$.   
\item The definition of $\F O_F$ makes sense for any $F\in \text{GL}_N(\C)$.  The additional condition $F \bar F = \pm 1$ is equivalent to the fact that $\hat U$ is always an irreducible unitary representation of $O^+_F$.  Indeed, Banica \cite{Ba0} showed that $\hat U$ is irreducible if and only if $F\bar F  \in \R 1$.  Moreover, $O^+_F$ and $O^+_{\lambda^{-1/2}F}$ are isomorphic as compact quantum groups for any $\lambda > 0$.   This is why we normalize so that  $F \bar F = \pm 1$.
\item It is also shown in \cite{Ba0} that $\F O_F$ is amenable if and only if $N =2$.  Since we are interested in non-amenable quantum groups, we will assume for the remainder of the paper that $N \ge 3$.
\item In \cite{BiDeVa}, it is shown that $\F O_F$ is unimodular if and only if $F$ is a scalar multiple of a unitary matrix.  As we shall assume  that $\F O_F$ is unimodular, and we normalize so that $F\bar{F} = \pm 1$, we may take $F$ to be unitary. 
\end{enumerate}
\end{rems}

\begin{notat}
In order to simplify our notation, for the remainder of this section we will denote the $L_p$-C$^\ast$-algebraic quantum group associated with $\F O_F$ by $(C^*_p(\F O_F),  \hat \Delta_p)$.
\end{notat}

The following theorem is the main result of the section and shows that the $L_p$-C$^\ast$-algebras of unimodular orthogonal free quantum groups provide examples of exotic quantum group C$^*$-algebras.  

\begin{thm} \label{thm:free_orthogonal}
Let $F\in \mc U_N$ with $F\bar F = \pm 1$. 
\begin{enumerate}
\item \label{surjective} For each $2 \le p < p' <\infty$, the canonical quotient map \[C^*_{p'}(\F O_F) \to C^*_p(\F O_F)\] extending the identity map on $L_1(\F O_F)$ is not injective.  
\item \label{non_iso} For each $2 <p < \infty$, the C$^\ast$-algebra $C^*_p(\F O_F)$ is not isomorphic (as a C$^\ast$-algebra) to either the universal C$^\ast$-algebra $C_u(O^+_F)$ or the reduced C$^\ast$-algebra $C(O^+_F)$. 
\end{enumerate}
\end{thm}

\begin{rem} \label{rem:exotic}
Theorem \ref{thm:free_orthogonal} shows that there exists an uncountable family of pairwise non-isomorphic compact quantum group completions $(C^*_p(\F O_F),  \hat\Delta_p)_{p \in (2,\infty)}$ of
$\Pol(O^+_F)$. Note, however, that it is an open question whether or not there is a C$^\ast$-algebra isomorphism $C^*_p(\F O_F)  \cong C^*_{p'}(\F O_F)$ when $p < p' \in (2,\infty)$.  All that we are able to  prove is that the canonical quotient map $C^*_{p'}(\F O_F) \to C^*_p(\F O_F)$ extending the identity map on $\Pol(O^+_F)$ is not an isomorphism.  In fact, at this time we know very little about the structure of the C$^\ast$-algebras $C^*_p(\F O_F)$, aside from the following theorem. 
\end{rem}

\begin{thm} \label{thm:uniquetrace}
Let $F\in \mc U_N$ with $F\bar F = \pm 1$.   Then for each $2\le p < \infty$, the Haar state is the unique tracial state on $C^*_p(\F O_F)$. 
\end{thm}

To prove theorems \ref{thm:free_orthogonal} and \ref{thm:uniquetrace},  we will need several technical results, which may be of independent interest.  In particular, we prove several quantum group analogues of results about the $\ell_p$-C$^\ast$-algebras of free groups proved by Okayasu in \cite{Ok12}.

\subsection{Preliminaries on orthogonal free quantum groups} \label{section:prelims_orthogonal}

Let $F \in \text{GL}_N(\C)$ with $F\bar F = \pm 1$.  Let us briefly review the representation theory of $O^+_F$ that was developed by Banica in \cite{Ba0}.  The collection   $\Irr(O^+_F)$ of equivalence classes of irreducible unitary representations of $O^+_F$ can be identified with $\N_0 = \N \cup \{0\}$  and therefore
 \[L_\infty(\F O_F) = \prod_{n \in \N_0} \mc B(H_n).\] Here, $H_0 = \C$  corresponds to the equivalence class of the trivial unitary representation of $O^+_F$ and $H_1 = \C^N$ corresponds to the fundamental unitary representation $\hat U^1= \hat U = [\hat u_{ij}] \in M_N(C_u(O^+_F))$.  For $n\ge 2$, one recursively defines the irreducible unitary representation $\hat U^n$ to be the unique subrepresentation of $\hat U^{\otop n}$ acting on $H_1^{\otimes n}$ that is not equivalent to any $\hat U^k$ for any $k < n$.  We denote by $\Pi_n \in \mc B(H_1^{\otimes n})$ the orthogonal projection onto the subrepresentation equivalent to  $\hat U^n$ and concretely identify $H_n$ with the subspace $\Pi_n(H_1^{\otimes n})$.  Under the above labeling of $\Irr(O^+_F)$, $O^+_F$ has the following fusion and conjugation rules (which are the same as those of  $SU(2)$):  \begin{align}\label{eqn:fusion}\overline{\hat U^n} \cong \hat U^n \quad \& \quad  \hat U^n \otop \hat U^k \cong \bigoplus_{r=0}^{\min\{n,k\}} \hat U^{n+k - 2r} \qquad (n,k \in \N_0).
\end{align}  Here, $\cong$ means unitary equivalence of unitary representations. In particular, \eqref{eqn:fusion} implies that for each $(l,k,n) \in \N_0^3$, there is at most one irreducible subrepresentation of $\hat U^n \otop \hat U^k$ unitarily equivalent to $\hat U^l$. In this case, we view $\hat U^l$ as a multiplicity-free subrepresentation of $\hat U^n \otop U^k$ in the obvious way.

Let $\rho > 1$ be such that $\rho+\rho^{-1} = N$.  Then we have the dimension formulas  \[d_n = \dim H_n = \frac{\rho^{n+1} - \rho^{-n-1}}{\rho - \rho^{-1}} = S_n(N),\]  where $(S_n)_{n \in \N_0}$ are the type 2 Chebyschev polynomials defined by 
\[
S_0(x) = 1, \ S_1(x) = x, \quad \& \quad S_{n+1}(x) = xS_n(x)- S_{n-1}(x) \qquad (n \ge 1). 
\]    

\subsection{The $L_q$-property of rapid decay for $\F O_F$} \label{section:q_Haagerup}

A crucial tool in our analysis of the C$^\ast$-algebras $C^*_p(\F O_F)$ for $2 \le p < \infty$ is an $L_q$-version of the property of rapid decay for $\F O_F$ (Proposition \ref{prop:qHaagerup}), where $1 < q \le 2$ is the conjugate exponent to $p$.  The $q = p =2$ version of this result was proved for $\F O_F$ by Vergnioux \cite{Ve}, which in turn is a quantum version of Haagerup's famous result \cite{Ha} showing that the free groups $\F_k$ have the property of rapid decay.   

We begin with a technical result.  In the following, let $\ell_2(d)$ be a $d$-dimensional complex Hilbert space and let $\{e_{ij}\}_{1 \le i,j \le d}$ be the canonical system of matrix units in $\mc B(\ell_2(d))$ (relative to a fixed orthonormal basis $(e_i)_i$ of $\ell_2(d)$).  Given an Hilbert space $H$ and $1 \le p \le \infty$, we also write $\mc S_p(H)$ for the ideal of Schatten-$p$-class operators.  We use the usual convention that $\mc S_\infty(H) = \mc B(H)$.  When $H = \ell_2(d)$, we simply write $\mc S_p^d$ for $\mc S_p(\ell_2(d))$. 

\begin{prop} \label{prop:Schatteninterpolation}
Let $H$ and $K$ be Hilbert spaces, $d \in \N$, and consider the linear map \begin{align*}&\Phi:\mc B(H \otimes \ell_2(d)) \otimes \mc B(\ell_2(d) \otimes K)  \to \mc B(H \otimes K);&\\
&\Phi:x \otimes y = \Big( \sum_{i,j=1}^d x_{ij}\otimes e_{ij}\Big) \otimes \Big(\sum_{k,l = 1}^d e_{kl} \otimes y_{kl}\Big) \mapsto \sum_{i,j=1}^d x_{ij} \otimes y_{ij} \qquad (x_{ij} \in \mc B(H), \ y_{kl} \in \mc B(K)). 
\end{align*}
Then for each $1 \le q \le 2$, we have \begin{align} \label{qineq}\|\Phi(x \otimes y)\|_{\mc S_q(H \otimes K)} \le \|x\|_{\mc S_q(H \otimes \ell_2(d))}\|y\|_{\mc S_q(\ell_2(d) \otimes K)} \qquad (x \in \mc S_q(H \otimes \ell_2(d)), \ y \in \mc S_q(\ell_2(d) \otimes K)).  \end{align}
\end{prop}

\begin{proof}
It suffices to prove \eqref{qineq} when $q =1$ and $q=2$. The result for general $1 \le q \le 2$ will then follow by complex interpolation for bilinear maps \cite[p. 96]{BeLo}.  

We first consider $q=2$.   In this case, \eqref{qineq} follows immediately from the triangle and Cauchy-Schwarz inequalities:
\begin{align*}
\|\Phi(x \otimes y)\|_{\mc S_2(H \otimes K)} & = \Big\|\sum_{i,j = 1}^d x_{ij} \otimes y_{ij}\Big\|_{\mc S_2(H \otimes K)} \le \sum_{i,j = 1}^d  \|x_{ij}\|_{\mc S_2(H)}  \|y_{ij}\|_{\mc S_2(K)} \\
& \le \Big(\sum_{i,j = 1}^d \|x_{ij}\|_{\mc S_2(H)}^2\Big)^{1/2} \Big(\sum_{i,j = 1}^d \|y_{ij}\|_{\mc S_2(K)}^2\Big)^{1/2}   \\
&= \|x\|_{\mc S_2(H \otimes \ell_2(d))}\|y\|_{\mc S_2(\ell_2(d) \otimes K)}.
\end{align*}

We now turn to the case $q = 1$.  Following \cite[Chapter 7]{Pi}, given an operator space $V$, we let $\mc S_1^d[V]$ be the Banach space obtained by equipping the vector space $M_d(V)$ with the norm $\|x\|_{\mc S_1^d[V]} = \inf\{\|a\|_{\mc \mc S_{2}^d}\|b\|_{\mc S_2^d}\|\hat x\|_{M_d(V)} \ | \ \hat x \in M_d(V), \ a,b \in \mc S_2^d, \ x = a\cdot\hat x \cdot b\}$.
Then we have, for any Hilbert space $L$, an isometric identification \[\mc S_1(L \otimes \ell_2(d)) \cong \mc S_1^d[\mc S_1(L)] \quad \text{given by} \quad x = \sum_{i,j = 1}^d\ x_{ij} \otimes e_{ij} \mapsto [x_{ij}] \in \mc S_1^d[\mc S_1(L)].\]
Here, $\mc S_1(L)$ is given its canonical operator space structure as the predual of $\mc B(L)$.  (We remark that the spaces $\mc S_1^d[V]$ are also considered in \cite[Section 4.1]{EfRu}, but are denoted $\mc T_d(V)$ there.) Now take $x = [x_{ij}] \in \mc S_1^d[\mc S_1(H)]$, $y = [y_{ij}] \in \mc S_1^d[\mc S_1(K)]$ and write $x = a \cdot \hat x \cdot b$ for some $a,b \in \mc S_2^d$, $\hat x = [\hat x_{ij}] \in M_d(\mc S_1(H))$.  Let $\tilde a \in M_{1, d^2}(\C)$ and $\tilde{b} \in M_{d^2,1}(\C)$
be the row and column matrices with entries given by $\tilde{a}_{1,ij} = a_{ji}$ and $\tilde{b}_{ij,1} = b_{ij}$ for each $1 \le i,j \le d$.  Then 
\begin{align*}\Phi(x \otimes y) &= \sum_{i,j,k,l=1}^d a_{ik}\hat x_{kl} \otimes y_{ij}b_{lj} = \tilde{a} \cdot(\hat x \otimes y) \cdot \tilde b,
\end{align*}
and it follows from the definition of the $\mc S_1(H \otimes K)$-norm (see \cite[Chapter 4]{Pi})  that 
\begin{align*}\|\Phi(x \otimes y)\|_{\mc S_1(H \otimes K)} &=\|\tilde{a} \cdot(\hat x \otimes y) \cdot \tilde b\|_{\mc S_1(H \otimes K)} \\
&\le \|\tilde a\|_{M_{1,d^2}} \|\tilde b\|_{M_{d^2, 1}} \|\hat x\|_{M_d(\mc S_1(H))} \|y\|_{M_d(\mc S_1(K))} \\
&= \|a\|_{\mc S_2^d}\|b\|_{\mc S_2^d}\|\hat x\|_{M_d(\mc S_1(H))} \|y\|_{M_d(\mc S_1(K))}.
\end{align*}
Taking the infimum of the above quantity over all representations $x = a \cdot \hat x \cdot b$, we obtain 
\[\|\Phi(x \otimes y)\|_{\mc S_1(H \otimes K)} \le \|x\|_{\mc S_1^d[\mc S_1(H)]}\|y\|_{M_d(\mc S_1(K))}.\]To conclude, we just observe that 
\begin{align*}
\|y\|_{\mc S_1^d[\mc S_1(K)]} &=  \inf\{\|a\|_{\mc S_{2}^d}\|b\|_{\mc S_2^d}\|\hat y\|_{M_d(\mc S_1(K))} \ | \ \hat y \in M_d(\mc S_1(K)), \ a,b \in \mc S_2^d, \ y = a\cdot\hat y \cdot b\} \\
& \ge  \inf\{\|a\|_{\mc B(\ell_2(d))}\|b\|_{\mc B(\ell_2(d))}\|\hat y\|_{M_d(\mc S_1(K))} \ | \ \hat y \in M_d(\mc S_1(K)), \ a,b \in \mc S_2^d, \  y = a\cdot\hat y \cdot b\} \\
& \ge \|y\|_{M_d(\mc S_1(K))}.
\end{align*}
\end{proof}

Recall that for a unimodular discrete quantum group $\G$, $\alpha_p:L_1(\G) \to \mc B(L_p(\G))$ denotes the left action of $L_1(\G)$ on $L_p(\G)$.  The following lemma should be compared with \cite[Lemma 1.3]{Ha} and \cite[Lemma 3.1]{Ok12} in the setting of free groups.  This result in the $q=2$ case is due to Vergnioux \cite[Section 4]{Ve}.    

\begin{lem} \label{lem:block_inequality}
Let $F\in \mc U_N$ with $F\bar F = \pm 1$.  Then there is a constant $D_N >0$ (depending only on $N$) such that for any $1 \le q \le 2$ and each triple $(l,k,n) \in \N_0^3$,  \[\|p_l (\alpha_q(\omega_x)(y))\|_{q} \le D_N \|x\|_{q}\|y\|_{q} \qquad (x \in \mc B(H_n), \  y \in \mc B(H_k)).\]
\end {lem}

\begin{proof}  In the following, we use the notation $\cdot$ to denote the natural left and right action of a von Neumann algebra on its predual.  We also recall that $d_n = \dim H_n$ for each $n \in \N_0$. 

Fix $x \in \mc B(H_n)$, $y \in \mc B(H_k)$ and $l$ as above.  Since $y \in \mf M_h$ and all of the interpolated actions $\alpha_q:L_1(\F O_F) \to \mc B(L_q(\F O_F))$  (by construction) agree on $\mf M_h$, it follows from Lemma \ref{lem:compatibility} that 
\[
p_l (\alpha_q(\omega_x)(y)) = p_l(\alpha_\infty(\omega_x)(y)) = p_l(x \star y) \in \mf M_h
~\mbox {and} ~ 
\omega_{p_l (x \star y)}  = p_l \cdot(\omega_x\star \omega_y). 
\]
 Without loss of generality, we  assume $H_l \subseteq H_n \otimes H_k$ is a subrepresentation.  (Otherwise $p_l \cdot(\omega_x\star \omega_y) = 0$ and there is nothing to prove.)  Since the inclusion $H_l \subseteq H_n \otimes H_k$ is multiplicity-free, we can choose a unique (up to multiplication by $\T$) isometric inclusion of representations $V_{l}^{n,k}: H_l \hookrightarrow H_n \otimes H_k$ and obtain
\begin{align} \label{eqn:convformula} p_l(x\star y) = \Big(\frac{d_nd_k}{d_l}\Big) (V_{l}^{n, k})^*(x \otimes y)V_{l}^{n,k}.
\end{align}
To verify equation \eqref{eqn:convformula}, we recall from \cite[Proposition 3.2]{PoWo} the coproduct formula
\begin{align*} 
(p_n \otimes p_k)\Delta(p_la) = V_l^{n,k}(p_la) (V_{l}^{n, k})^* \qquad (a \in L_\infty(\F O_F)),
\end{align*}
and the formula  
\[p_l \cdot (\omega_{p_n} \star \omega_{p_k}) =  \Big(\frac{d_nd_k}{d_l}\Big) \omega_{p_l},\] which was obtained by Vergnioux in his proof of \cite[Lemma 4.6]{Ve}. 
Then we have, for each $a \in L_\infty(\F O_F)$,
\begin{align*}
\langle \omega_{p_l(x \star y)}, a\rangle&=\langle p_l \cdot (\omega_x \star \omega_y) , a \rangle \\
&= (h \otimes h) \Big(\Delta(p_lap_l)(x \otimes y)\Big) \\
&=  (h \otimes h) \Big((p_n \otimes p_k)\Delta(p_la)(x \otimes y)(p_n \otimes p_k)\Delta(p_l)\Big) \\
&= (h \otimes h) \Big(V_{l}^{n,k}(p_la)(V_{l}^{n,k})^*(x \otimes y)V_{l}^{n,k}p_l(V_{l}^{n,k})^*\Big) \\
&= (h \otimes h) \Bigg(\Delta\Big(p_la(V_{l}^{n,k})^*(x \otimes y)V_{l}^{n,k}\Big)(p_n \otimes p_k)\Bigg)\\
&= \langle p_l \cdot (\omega_{p_n} \star \omega_{p_k}), a\big((V_{l}^{n,k})^*(x \otimes y)V_{l}^{n,k}\big)  \rangle \\
&=\Big(\frac{d_nd_k}{d_l}\Big)\langle \omega_{(V_{l}^{n,k})^*(x \otimes y)V_{l}^{n,k}}, a  \rangle,
\end{align*}
which gives \eqref{eqn:convformula}.

Therefore, to prove the lemma, we must find a constant $D_N > 0$ (independent of $(k,l,n) \in   \N_0^3$) such that 
\begin{align*}\|p_l(x\star y)\|_q = \Big(\frac{d_nd_k}{d_l}\Big) \Big\|(V_{l}^{n,k})^*(x \otimes y)V_{l}^{n,k}\Big\|_q \le D_N\|x\|_q\|y\|_q \qquad (x \in \mc B(H_n), \ y \in \mc B(H_k), \ q \in [1,2]). \end{align*}
Equivalently, writing the above $L_q$-norms in terms of Schatten $q$-norms, we require $D_N$ to satisfy \begin{align} \label{eqn:schatten_ineq}\Big\|(V_{l}^{n,k})^*(x \otimes y)V_{l}^{n,k}\Big\|_{\mc S_q(H_l)} \le D_N\Big(\frac{d_l}{d_n d_k}\Big)^{1-1/q}\|x\|_{\mc S_q(H_n)}\|y\|_{\mc S_q(H_k)},
\end{align}
for each $x \in \mc B(H_n)$, $y \in \mc B(H_k)$ and $q \in [1,2]$.

Let $\{e_i\}_{i}$ be a fixed orthonormal basis for $H_1 = \C^N$.
Since  $F \in {\mathcal U}_N$ is a unitary matrix,  $\{f_i \}_i= \{Fe_i\}_i $ 
determines  another orthonormal basis for $H_1$.
Let $\{e_{ij}\}$ (resp. $\{f_{ij}\}$) denote the standard matrix units for $\mc B(H_1)$ relative to the orthonormal basis $\{e_i\}_i$ (resp. $\{f_i \}_i$).  
Given functions $i,j:\{1,\ldots, r\} \to \{1, \ldots, N\}$, we will also write 
$e_i$ (resp. $f_i$) for the tensors 
$e_{i(1)}\otimes \ldots \otimes e_{i(r)}  \in H_1^{\otimes r}$ 
(resp. $f_{i(1)}\otimes \ldots \otimes f_{i(r)}\in H_1^{\otimes r}$), and write 
$\{e_{ij}\}$ (resp. $\{f_{ij}\}$) for the matrix units $\{e_{i(1)j(1)}\otimes 
\ldots \otimes e_{i(r)j(r)} \} $  
(resp. $\{f_{i(1)j(1)}\otimes \ldots \otimes f_{i(r)j(r)}\}$) in $\mc B(H_1^{\otimes r})$.  
Since $H_l \subseteq H_n \otimes H_k$ is a subrepresentation, there is a unique 
$0 \le r \le \min\{n,k\}$ such that $l = k+n - 2r$.  Finally, let us identify $H_l$ with 
the highest weight subrepresentation 
$\Delta(p_l)(H_{n-r} \otimes H_{k-r})$ of $H_{n-r} \otimes H_{k-r}$, and similarly identify $H_n$, $H_k$ with the highest weight subrepresentations of $H_{n-r} \otimes H_{1}^{\otimes r}$ and $H_{1}^{\otimes r} \otimes H_{k-r}$, respectively.  With these identifications, we can uniquely write \[x = \sum_{i,j} x_{i,j}\otimes e_{ij}, \qquad y = \sum_{i,j} f_{ij} \otimes y_{i,j}, \]
where $i,j:\{1,\ldots, r\} \to \{1, \ldots, N\}$ are functions, $x_{i,j} \in \mc B(H_{n-r})$, and $y_{i,j} \in \mc B(H_{k-r})$.  Using the description of the isometry $V_{l}^{n,k}$ given by \cite[Equations $(7.1)$--$(7.3)$]{VaVe}, it follows that
\[
(V_{l}^{n,k})^*(x \otimes y)V_{l}^{n,k} =\Big( \frac{1}{d_r}\prod_{s=1}^r\frac{d_s d_{n-r+s-1}d_{k-r+s-1}}{d_{l+s}d_{s-1}^2} \Big) \Delta(p_l)\Big(\sum_{i,j} x_{i,j} \otimes y_{\check{i},\check{j}}\Big)\Delta(p_l),
\] 
where for each multi-index $i = (i(1), \ldots,i(r-1), i(r))$, we put $\check{i}:= (i(r), i(r-1), \ldots, i(1))$.

Using the fact that $d_n = \frac{\rho^{n+1}-\rho^{-n-1}}{\rho-\rho^{-1}}$, where $\rho > 1$ is given by $N = \rho + \rho^{-1}$, one easily finds a constant $\tilde D_N > 0$ (independent of $(n,k,l)$) such that 
\begin{align} \label{eqn:dim_ineq}
\frac{1}{d_r}\prod_{s=1}^r\frac{d_s d_{n-r+s-1}d_{k-r+s-1}}{d_{l+s}d_{s-1}^2}  \le \tilde D_N\Big(\frac{d_l}{d_nd_k}\Big)^{1/2}.
\end{align}
To conclude the proof of the lemma, we show that inequality \eqref{eqn:schatten_ineq} holds with the constant $D_N = \tilde D_N$.  Indeed, if we take the Schatten $q$-norm of $(V_{l}^{n,k})^*(x \otimes y)V_{l}^{n,k}$ using inequality \eqref{eqn:dim_ineq} and apply Proposition \ref{prop:Schatteninterpolation} to the result, we obtain 
\begin{align*}&\Big\|(V_{l}^{n,k})^*(x \otimes y)V_{l}^{n,k}\Big\|_{\mc S_q(H_l)} \\
&\le \tilde D_N\Big(\frac{d_l}{d_n d_k}\Big)^{1/2} \Big\|\Delta(p_l)\Big(\sum_{i,j} x_{i,j} \otimes y_{\check{i},\check{j}}\Big)\Delta(p_l)\Big\|_{\mc S_q(H_l)}\\
&\le  \tilde D_N\Big(\frac{d_l}{d_n d_k}\Big)^{1/2} \Big\|\sum_{i,j} x_{i,j} \otimes y_{\check{i},\check{j}}\Big\|_{\mc S_q(H_{n-r} \otimes H_{k-r})}\\
&\le  \tilde D_N\Big(\frac{d_l}{d_n d_k}\Big)^{1/2}\Big\|\sum_{i,j} x_{i,j}\otimes e_{ij}\Big\|_{\mc S_q(H_{n-r} \otimes H_1^{\otimes r})}\Big\| \sum_{i,j} e_{ij} \otimes y_{\check{i},\check{j}}\Big\|_{\mc S_q(H_1^{\otimes r} \otimes H_{k-r})} \\
&=\tilde D_N\Big(\frac{d_l}{d_n d_k}\Big)^{1/2}\|x\|_{\mc S_q(H_n)} \Big\| \sum_{i,j} e_{ij} \otimes y_{\check{i},\check{j}}\Big\|_{\mc S_q(H_1^{\otimes r} \otimes H_{k-r})} \\
&=\tilde D_N\Big(\frac{d_l}{d_n d_k}\Big)^{1/2}\|x\|_{\mc S_q(H_n)} \| y\|_{\mc S_q(H_{k})}. 
\end{align*}
In the last equality we have used the fact that $\| y\|_{\mc S_q(H_{k})} = \Big\| \sum_{i,j} e_{ij} \otimes y_{\check{i},\check{j}}\Big\|_{\mc S_q(H_1^{\otimes r} \otimes H_{k-r})}$.  This follows because $\sum_{i,j} e_{ij} \otimes y_{\check{i},\check{j}} = (v^* \otimes \iota)y (v \otimes \iota)$, where $v\in \mc U(H_1^{\otimes r})$ is the  unitary defined by $\sigma(f_i) = e_{\check{i}}$.  Finally, since $\Big(\frac{d_l}{d_n d_k}\Big)^{1/2} \le \Big(\frac{d_l}{d_n d_k}\Big)^{1-1/q}$, the proof is complete.
\end{proof}

Combining Lemma \ref{lem:block_inequality} with a standard H\"older estimate, we finally obtain our $L_q$-property of rapid decay for $\F O_F$. 

\begin{prop} \label{prop:qHaagerup}
Let $F\in \mc U_N$ with $F\bar F = \pm 1$.  Then there is a constant $D_N > 0$ depending only on $N$ such that for each $1 \le q \le 2$ and each $n \in \N_0$, \[\|\alpha_q(\omega_x)\|_{\mc B(L_q(\F O_F))} \le D_N(n+1)\|x\|_q \qquad (x \in \mc B(H_n)).\] 
\end{prop}

\begin{proof}
Given $n,k,l \in \N_0$, we use the notation $l \subset  n \otimes k$ to mean that $H_l$ is a subrepresentation of $H_n \otimes H_k$.  Now fix $x \in \mc B(H_n)$ and $y = \sum_{k \in \N_0} y_k \in L_q(\F O_F)$, where $y_k = p_k y \in \mc B(H_k)$.  By Lemma \ref{lem:block_inequality}, we have
\begin{align*}
\|\alpha_q(\omega_x)(y)\|_q^q &= \sum_{l \in \N_0} \|p_l\alpha_q(\omega_x)(y)\|_q^q  = \sum_{l \in \N_0} \Big\|\sum_{k: l \subset n \otimes k} p_l(\alpha_q(\omega_x)(y_k))\Big\|_q^q \\
&\le D_N^q \|x\|_q^q \sum_{l \in \N_0}\Big(\sum_{k: l \subset n \otimes k} \|y_k\|_q\Big)^q.
\end{align*}
Now let  $2 \le p < \infty$ be the conjugate exponent to $q$.  Applying H\"older's inequality to the internal sum above we then have 
\begin{align*}
\sum_{l \in \N_0}\Big(\sum_{k: l \subset n \otimes k} \|y_k\|_q\Big)^q &\le \sum_{l \in \N_0}\Big(\sum_{k: \ l \subset n \otimes k} \|y_k\|_q^q\Big) \Big(\sum_{k: \ l \subset n \otimes k} 1 \Big)^{\frac{q}{p}} \\
&\le \sup_{l \in \N_0}  \Big(\sum_{k: \ l \subset n \otimes k} 1 \Big)^{\frac{q}{p}}  \sum_{l \in \N_0}\Big(\sum_{k: \ l \subset n \otimes k} \|y_k\|_q^q\Big) \\
&=  \sup_{l \in \N_0} \Big(\sum_{k: \ l \subset n \otimes k} 1 \Big)^{\frac{q}{p}} \sum_{k \in \N_0}\|y_k\|_q^q \Big(\sum_{l: \ l \subset n \otimes k} 1\Big) \\
&\le  \sup_{l \in \N_0}  \Big(\sum_{k: \ l \subset n \otimes k} 1 \Big)^{\frac{q}{p}}  \sup_{k \in \N_0}\Big(\sum_{l: \ l 
\subset n \otimes k} 1\Big) \|y\|_q^q  .
\end{align*}
Since the fusion rules for $\F O_F$ force both of the sets $\{k: l \subset n \otimes k \}$ and $\{l: l \subset n \otimes k \}$ to have cardinality at most $n +1$, we obtain
\[\|\alpha_q(\omega_x)y\|_q^q \le D_N^q(n+1)^{\frac{q}{p} + 1}\|x\|_q^q \|y\|_q^q.\] The proposition now follows.  
\end{proof}

We now combine Proposition \ref{prop:qHaagerup} with Lemma \ref{lem:Kac_result} to obtain a corresponding Haagerup-type inequality for $C^*_p(\F O_F)$.

\begin{prop} \label{prop:Lqestimate_operatornorm}
Let $F\in \mc U_N$ with $F\bar F = \pm 1$.  Then there is a constant $D_N > 0$ depending only on $N$ such that for each $2 \le p < \infty$ and each  $n \in \N_0$, 
\[\|\omega_x\|_{C^*_p(\F O_F)} \le D_N(n+1)\|x\|_q \qquad (x \in \mc B(H_n)),\] where $1 <q \le 2$ is the conjugate exponent to $p$. 
\end{prop}

\begin{proof}
Fix $n \in \N_0$ and $x \in\mc B(H_n)$.  Let $\pi:L_1(\F O_F) \to \mc B(H)$ be an $L_p$-representation with canonical dense subspace $H_0$.  By density, it suffices to show that \[\|\pi(\omega_x)\xi\| \le D_N (n+1)\|x\|_q\] for each unit vector $\xi \in H_0$.  Fix  such a $\xi \in H_0$ and consider the cyclic representation $(\sigma, K, \xi)$ given by $\sigma = \pi|_K$ and $K = \overline{\pi(L_1(\F O_F))\xi}$.  Since the coefficient function $\varphi^\sigma_{\xi,\xi}$ belongs to $L_p(\G)$, Lemma \ref{lem:Kac_result} implies that $\sigma$ is an $L_p$-representation and
\[
\|\pi(\omega_x)\xi\| \le \|\sigma(\omega_x)\| \le \liminf_{k \to \infty} \|(\omega_x^\sharp \star \omega_x)^{\star 2k}\|_q^{\frac{1}{4k}}. 
\]
 On the other hand, from the definition of the left action $\alpha_q:L_1(\F O_F) \to \mc B(L_q(\F O_F))$, \begin{align*}
\|(\omega_x^\sharp \star \omega_x)^{\star 2k}\|_q &=  \|(\alpha_q(\omega_x^\sharp)\alpha_q(\omega_x))^{2k-2}\alpha_q(\omega_x^\sharp)(x)\|_{q} \\
& \le \|\alpha_q(\omega_x^\sharp)\|_{\mc B(L_q(\F O_F))}^{2k}\|\alpha_q(\omega_x)\|_{\mc B(L_q(\F O_F))}^{2k-1}\|x\|_q. 
\end{align*}  To bound the right hand side of the above inequality, observe that $\omega_x^\sharp = \omega_{S(x^*)}$, where $S$ is the antipode of $\F O_F$.  Indeed, for each $y \in L_\infty(\F O_F)$ we have \[\langle \omega_x^\sharp, y \rangle = \overline{h(S(y)^*x)} = h(x^*S(y)) = h(yS(x^*)) = \langle \omega_{S(x^*)}, y \rangle.\]  Moreover, we have $S(x^*) \in \mc B(H_n)$ and  $\|S(x^*)\|_q = \|x\|_q$.  Therefore we may apply Proposition \ref{prop:qHaagerup} to get \[ \|\alpha_q(\omega_x^\sharp)\|_{\mc B(L_q(\F O_F))}^{2k}\|\alpha_q(\omega_x)\|_{\mc B(L_q(\F O_F))}^{2k-1}\|x\|_q \le (D_N(n+1))^{4k-1}\|x\|_q^{4k}.\]  The proposition now follows from these inequalities.  
\end{proof}

The following is an immediate consequence of Proposition \ref{prop:Lqestimate_operatornorm} and H\"older's inequality. For each non-zero $x \in C_c(\F O_F)$, we define $n(x):=\max\{n \in \N_0: p_n x \ne 0\}$.

\begin{cor} \label{cor:Lqestimate_operatornorm}
For each $2 \le p < \infty$, we have   
\[\|\omega_x\|_{C^*_p(\F O_F)} \le D_N (n(x)+1)^{1+1/p}\|x\|_q \qquad (x \in C_c(\F O_F)),\]
where $1 <q \le 2$ is the conjugate exponent to $p$. 
\end{cor}

\begin{proof}
Let $x \in C_c(\F O_F)$ and write $x = \sum_{n=0}^{n(x)}p_nx$.  By the triangle inequality, Proposition \ref{prop:Lqestimate_operatornorm}, and H\"older's inequality, we obtain
\begin{align*}\|\omega_x\|_{C^*_p(\F O_F)} &\le \sum_{n=0}^{n(x)}\|\omega_{p_nx}\|_{C^*_p(\F O_F)} \le \sum_{n=0}^{n(x)}D_N(n+1)\|p_nx\|_{q} \\
&\le D_N\Big(\sum_{n=0}^{n(x)}(n+1)^p\Big)^{1/p}\|x\|_q \le D_N(n(x)+1)^{1 + 1/p}\|x\|_q.
\end{align*}
\end{proof}

\subsection{Positive definite functions associated with $C^*_p(\F O_F)$}

In this section we use our Haagerup inequality for $C^*_p(\F O_F)$ (Proposition \ref{prop:qHaagerup}) to give a characterization of the positive definite functions on a unimodular orthogonal free quantum group $\F O_F$ that extend to states on $C^*_p(\F O_F)$ in terms of a weak $L_p$-condition relative to a natural length function $\ell$ defined on $\Irr(\widehat {\F O_F})$.  We begin by recalling $\ell$ and the natural semigroup that it generates. 

\begin{notat}
Denote by $\ell:\N_0 \cong \Irr(\widehat {\F O_F}) \to [0,\infty)$ the canonical length function.  I.e., $\ell(n) = n$ for each $n \in \N_0$.  For each $F \in \text{GL}_N(\C)$ with $F \bar F = \pm 1$, we define a multiplicative semigroup of central elements $(r^\ell)_{r \in (0,1)} \subset C_0(\F O_F)$ by setting
\[r^\ell = \prod_{n \in \N_0} r^n 1_{\mc B(H_n)} \in C_0(\F O_F).\]  At times, we will also regard $\ell$ as the unbounded operator $\ell = \prod_{n \in \N_0} n 1_{\mc B(H_n)}$
affiliated to $L_\infty(\F O_F)$ which generates the semigroup $r^\ell$.   
\end{notat}

\begin{rem} \label{rem:pd}
The semigroup $(r^\ell)_{r \in (0,1)} \in C_0(\F O_F)$ can be regarded as an analogue of the Poisson semigroup of positive definite functions on the free group $\F_N$ (see \cite{Ha}, Lemma 1.2).
Note, however, that unlike in the free group case, $r^\ell$ fails to be a positive definite function on $\F O_F$ (see Section 1 of \cite{Ve}).

For the purposes of the following theorem, we will need to use the fact that the quantum groups $\F O_F$ have the Haagerup property.  More precisely, we require an approximate identity  $(\varphi_r)_{r \in (0,1)} \subset C_0(\F O_F)$ consisting of positive definite functions which has the same (exponential) decay rate as the semigroup $(r^\ell)_{r \in (0,1)}$.  This can be done as follows.  Let $(S_n)_{n \in \N_0}$ be the type 2 Chebyschev polynomials defined in Section \ref{section:prelims_orthogonal}, and for each $r \in (0,1)$, define 
\[
\varphi_{r} = \prod_{n \in \N_0} \frac{S_n(rN)}{S_n(N)}1_{\mc B(H_n)} \in L_\infty(\F O_F).
\] Let $2 <t_0 < 3$ be fixed.  Using the formula $S_n(x) = \frac{\rho(x)^n(1-\rho(x)^{-2n-2})}{1- \rho(x)^{-2}}$ where $\rho(x) = \frac{x}{2}\Big(1 + \sqrt{1 -4x^{-2}}\Big)$ and $x \ge 2$, one can easily find constants $0 < C_1 < C_2$ depending only on $N$ and $t_0$ such that \[C_1r^n \le \frac{S_n(rN)}{S_n(N)} \le C_2r^n \qquad (t_0N^{-1} \le r < 1). \]
Since $\lim_{r \to 1}\frac{S_n(rN)}{S_n(N)} =1$ for each $n \in \N_0$, this shows that  $(\varphi_r)_{r \in (0,1)} \subset C_0(\F O_F)$ is an approximate unit with the required exponential decay. To see that $\varphi_r$ is positive definite, we appeal to  \cite[Proposition 4.4]{BrAP} and \cite[Theorem 17]{DeFrYa}, where it is shown that each $\varphi_r$ implements a unital completely positive (central) multiplier of $L_1(O^+_F)$.  I.e., there exists a completely bounded map $L_* \in \mc {CB}(L_1(O^+_F))$ given by \[\hat \lambda (L_* \hat \omega) = \varphi_r \hat \lambda (\hat \omega) \qquad (\hat \omega \in L_1(O^+_F)),\] whose adjoint $L = (L_*)^*$ is a $\sigma$-weakly continuous unital completely positive map on $L_\infty(O^+_F)$.  (This was proved for $F = 1_N$ in \cite{BrAP} and the general case can be deduced from \cite{DeFrYa}.)  To conclude, we just note that the complete positivity of $L$ is equivalent to the positive definiteness of  $\varphi_r^* = \varphi_r$ by Theorem 6.4 of \cite{Da11}. See also Remark \ref{rem:pd}.   
\end{rem}

We now arrive at our characterization of the positive definite functions on $\F O_F$ which extend to states on $C^*_p(\F O_F)$.  

\begin{thm}\label{thm:characterization_p_continuity}
Let $F\in \mc U_N$ with $F\bar F = \pm 1$ and let $2 \le p < \infty$.  The following conditions are equivalent for a positive definite function $\varphi \in B(\F O_F)$: 
\begin{enumerate}
\item \label{FS} $\varphi \in B_{L_p(\F O_F)}(\F O_F)$.
\item \label{sup} $\sup_{n \in \N_0} (n+1)^{-1}\|p_n\varphi\|_{p} < \infty$.
\item \label{length} $(1 + \ell)^{-1 - \frac{2}{p}}\varphi \in L_p(\F O_F)$.
\item \label{weaklp} $\varphi$ is weakly $L_p$.  I.e., $r^\ell \varphi \in L_p(\F O_F)$ for each $0 < r < 1$.
\end{enumerate}
\end{thm}

\begin{proof} In the following, we assume without loss of generality that $\|\varphi\|_{B(\F O_F)} = 1$. 

\eqref{FS} $\implies$ \eqref{sup}.  Fix $n \in \N_0$ and let $x = p_n|\varphi|^{p-2}\varphi^* \in C_c(\F O_F)$.  Then we have 
\[\|p_n\varphi\|_{p}^p = h(p_n|\varphi|^p) = h(\varphi p_n|\varphi|^{p-2}\varphi^*) = h(\varphi x) = \langle \omega_x,\varphi \rangle.\]
In particular, by Proposition \ref{prop:Lqestimate_operatornorm}, there is a constant $D_N > 0$ (independent of $n \in \N_0$) such that 
\begin{align*}\|p_n\varphi\|_{p}^p&= \langle \omega_x,\varphi \rangle \le \|\omega_x\|_{C^*_p(\F O_F)} \le D_N(n+1)\|x\|_q \qquad (\text{where }p^{-1} + q^{-1} = 1)  \\
&=D_N(n+1) h(|x|^{\frac{p}{p-1}})^{\frac{p-1}{p}} = D_N(n+1) h\big(\big| p_n|\varphi|^{p-2}\varphi^*\big|^{\frac{p}{p-1}}\big)^{\frac{p-1}{p}} \\
&= D_N(n+1) h(p_n(|\varphi|^{p-1})^{\frac{p}{p-1}})^{\frac{p-1}{p}}\\
&=D_N(n+1) \|p_n\varphi\|_p^{p-1}.
\end{align*}
From these inequalities, we obtain $(n+1)^{-1}\|p_n\varphi\|_p \le D_N$.

\eqref{sup} $\implies$ \eqref{length}.  This implication follows from the inequality
\begin{align*}
\|(1 + \ell)^{-1 - \frac{2}{p}}\varphi\|_p^p &= \sum_{n \in \N_0} (1+n)^{-p-2} \|p_n\varphi\|_p^p\\
&\le \sup_{n \in \N_0}\Big\{(n+1)^{-1}\|p_n \varphi\|_p\Big\}^p \Big(\sum_{n \in \N_0} (n+1)^{-2}\Big). 
\end{align*}

\eqref{length} $\implies$ \eqref{weaklp}. For each $r \in (0,1)$, one can find a constant $C_r >0$ such that $r^{pn} \le C_r(n+1)^{-p-2}$ for all $n \in \N_0$.  From this, one easily sees that $\|r^\ell \varphi\|_p^p \le C_r\|(1 + \ell)^{-1 - \frac{2}{p}}\varphi\|_p^p.$ 

\eqref{weaklp} $\implies$ \eqref{FS}.  For each $r \in (0,1)$, let $\varphi_r \in B(\F O_F)$ be the positive definite function defined in Remark \ref{rem:pd}.  In particular, for $r > t_0N^{-1}$,  $0 \le \varphi_r \le C_2r^{\ell}$, which by \eqref{weaklp} and Lemma \ref{lem:Kac_result} implies that the norm one positive definite function $\varphi_r\varphi$ belongs to $L_p(\F O_F) \cap A_{L_p(\F O_F)}(\F O_F)$.  Since $\lim_{r \to 1} p_n\varphi_r\varphi = p_n \varphi$ for each $n \in \N_0$, it follows by linearity that  
\begin{align*}
|\langle \omega_x, \varphi \rangle| = \lim_{r \to 1}|\langle \omega_x, \varphi_r \varphi \rangle|\le \|\omega_x\|_{C^*_p(\F O_F)} \qquad (x \in C_c(\F O_F)),
\end{align*} 
In particular,  $\varphi \in B_{L_p(\F O_F)}(\F O_F)$.
\end{proof}

\subsection{Proof that $C^*_p(\F O_F)$ is exotic for $2 < p < \infty$}

Before proceeding to the proof of Theorem \ref{thm:free_orthogonal}, we need the following lemma.  Recall that $(S_n)_{n \in \N_0}$ are the Chebyschev polynomials defined in Section \ref{section:prelims_orthogonal}.  We also remind the reader that $\dim(H_n) = S_n(N)$, where $H_n$ is the irreducible unitary representation of $O^+_F$ with label $n$.

\begin{lem} \label{lem:ratiotest}
Let  $p \in (1,\infty)$, $N \ge 3$ and $\rho>1$ be such that $\rho + \rho^{-1} = N$.  Then the series \[\sum_{n \in \N_0} r^{pn}S_n(N)^2  \qquad (0 < r < 1)\] 
converges if $r < \rho^{-2/p}$ and diverges if $r > \rho^{-2/p}$.
\end{lem}

\begin{proof}
By the ratio test, this series converges if $L:= \limsup_{n}\frac{r^pS_{n+1}(N)^2}{S_n(N)^2} < 1$ and diverges if $L > 1$.  Since
$S_n(N) = \rho^{n}\Big(\frac{1- \rho^{-2n-2}}{1-\rho^{-2}}\Big)$,  we obtain
\begin{align*}
L  = \limsup_n r^p\rho^{2}\frac{(1-\rho^{-2n-4})^2}{(1-\rho^{-2n-2})^2} = r^p\rho^{2}. 
\end{align*}
\end{proof}

\begin{proof}[Proof of Theorem \ref{thm:free_orthogonal}]
\eqref{surjective} Let $2 \le p < p' < \infty$.  Since the canonical quotient map $C^*_{p'}(\F O_F) \to C^*_p(\F O_F)$ is injective if and only if $B_{L_p(\F O_F)}(\F O_F) = B_{L_{p'}(\F O_F)}(\F O_F)$, it suffices to exhibit a positive definite element $\varphi \in B_{L_{p'}(\F O_F)}(\F O_F)\backslash B_{L_p(\F O_F)}(\F O_F)$.  Consider the net $(\varphi_r)_{r \in (0,1)}$ of positive definite elements of $B(\F O_F)$ introduced in Remark \ref{rem:pd}.  Choose $0 < r_0 < 1$ so that $\rho^{-2/p} < r_0 < \rho^{-2/p'}$, where $\rho>1$ is such that $\rho + \rho^{-1} = N$, and consider the corresponding positive definite element $\varphi_{r_0}$.  Since there are constants $0 <C_1 <C_2$ such that $C_1r_0^\ell \le \varphi_{r_0} \le C_2r_0^\ell$, we have, for each $r \in (0,1)$,
\[\|r^\ell \varphi_{r_0}\|^{p'}_{p'} < \infty \iff \|(rr_0)^\ell\|^{p'}_{p'} =\sum_{n \in \N_0} (rr_0)^{p'n}S_n(N)^2 < \infty.\]  
Since $r_0 < \rho^{-2/p'}$, $\varphi_{r_0}$ is weakly $L_{p'}$ by Lemma \ref{lem:ratiotest} and $\varphi_{r_0} \in B_{L_{p'}(\F O_F)}(\F O_F)$ by Theorem \ref{thm:characterization_p_continuity}.  On the other hand, since $\rho^{-2/p} < r_0$, there exists an $r_1 \in (0,1)$ such that $\rho^{-2/p}  < r_1r_0$, and therefore $\|r_1^\ell \varphi_{r_0}\| \ge C_1\|(r_1r_0)^\ell\|_{p} = \infty$ by Lemma \ref{lem:ratiotest}.  In particular, $\varphi_{r_0} \notin B_{L_p(\F O_F)}(\F O_F)$ by Theorem \ref{thm:characterization_p_continuity}.

\eqref{non_iso}.  Fix $2 < p < \infty$.  To show that $C^*_{p}(\F O_F) \ncong C_u(O^+_F)$, it suffices to show that $C^*_p(\F O_F)$ does not have any non-trivial one-dimensional representations (characters).  Suppose, to reach a contradiction, that $\pi:C^*_p(\F O_F) \to \C = \mc B(\C)$ is a character and let $U_\pi \in \mc U(L_\infty(\F O_F))$ be the unitary element implementing $\pi$.  Since $\pi$ is one-dimensional, $U_\pi = \varphi^\pi_{1,1}$ is also a positive definite coefficient function in $B_{L_p(\F O_F)}(\F O_F)$ and  Theorem \ref{thm:characterization_p_continuity} \eqref{weaklp} implies that \[\|r^\ell U_\pi\|_{p}^p  = \sum_{n \in \N_0} r^{np} \|p_n U_\pi\|_p^p = \sum_{n \in \N_0} r^{np}S_n(N)^2< \infty \qquad (0 < r < 1). \]
This contradicts Lemma \ref{lem:ratiotest}.  The non-isomorphism $C^*_{p}(\F O_F) \ncong C(O^+_F)$ follows from the fact that $C(O^+_F)$ is a simple C$^\ast$-algebra \cite[Theorem 7.2]{VaVe}, while $C^*_{p}(\F O_F)$ is not simple by \eqref{surjective}. 
\end{proof}

\subsection{Uniqueness of trace} \label{section:uniquetrace}

In this section we prove Theorem \ref{thm:uniquetrace}, which states that the Haar state is the unique tracial state on $C^*_p(\F O_F)$ for each $2 \le p < \infty$.  The $p=2$ case was proved in \cite{VaVe} using a quantum analogue of the ``conjugation by generators'' approach to simplicity of reduced free group C$^\ast$-algebras.  In our proof, we develop an $L_p$ version of this quantum conjugation by generators method.

In $\mc B(H_1) = \mc B(\ell_2(N)) \subset L_\infty(\F O_F)$, let $(e_{ij})_{1 \le i,j \le N}$ be a system of matrix units and let $(\omega_{ij})_{1 \le i,j \le N} \subset L_1(\F O_F)$ be the corresponding dual basis (relative to the duality induced by the Haar weight $h$).  Note that  $L_1(\F O_F)$ is generated by $(\omega_{ij})_{1 \le i,j \le N}$ as a Banach $\ast$-algebra and that the matrix $\Omega:=[\omega_{ij}]$ is a unitary element of the $\ast$-algebra $M_N(L_1(\G))$.  This just follows from the fact that $(\lambda \otimes \iota)\Omega = \hat U^* \in M_N(C(O^+_F))$, where $\hat U$ is the fundamental representation of $O^+_F$.  Using the family $(\omega_{ij})_{1 \le i,j \le N}$, we define the following ``conjugation by generators'' map 
\begin{align} \label{eqn:conjugation}
\Phi: L_1(\F O_F) \to L_1(\F O_F); \qquad \Phi(\omega) = \frac{1}{2N} \sum_{i,j =1}^N \big( \omega_{ij}\star \omega \star \omega_{ij}^\sharp + \omega_{ij}^\sharp\star \omega \star \omega_{ij}\big).
\end{align} 

\begin{rem}
Since $\Omega = [\omega_{ij}]$ is unitary, we have $\sum_{i,j} \omega_{ij}\star \omega_{ij}^\sharp = \sum_{i,j} \omega_{ij}^\sharp\star \omega_{ij} = N$ and it follows that $\Phi$ is unital and $\tau = \tau \circ \Phi$ for any tracial state $\tau:L_1(\F O_F) \to \C$.  
\end{rem}    
Since we have $\omega_{ij} = \frac{1}{N}\omega_{e_{ji}}$, it follows that $\|\omega_{ij}\|_1 = \|\omega_{ij}^\sharp\|_1 = 1$ for each $1 \le i,j \le N$, and therefore one obtains the trivial bound $\|\Phi\|_{L_1(\F O_F) \to L_1(\F O_F)} \le N$.  In the following lemma, we show that this bound can be greatly improved.
\begin{lem} \label{lem:L_1contractive}
$\|\Phi\|_{L_1(\F O_F) \to L_1(\F O_F)} = 1$.
\end{lem}

\begin{proof}
Since $L_1(\F O_F) = L_\infty(\F O_F)_*$ is a completely contractive Banach algebra, we have 
\[2N\|\Phi(\omega)\| \le \Big\| \sum_{i,j=1}^N \omega_{ij}\star \omega \otimes \omega_{ij}^\sharp\Big\|_{L_1(\F O_F) \widehat{\otimes} L_1(\F O_F)} +\Big\|\sum_{i,j=1}^N \omega_{ij}^\sharp \otimes \omega\star \omega_{ij}\Big\|_{L_1(\F O_F) \widehat{\otimes} L_1(\F O_F)}. \] It therefore suffices to show that the two norms on the right side of the above equation are bounded by $N\|\omega\|_1$ for each $\omega \in L_1(\F O_F)$.
 
Let $(e_i)_{1 \le i \le N}$ be the orthonormal basis corresponding to the matrix units $e_{ij} \in \mc B(H_1)$ and put $a = [e_1^t \ e_2^t \ldots \ e_N^t] \in M_{1,N^2}$.  Then we can write
$\sum_{i,j=1}^N \omega_{ij}\star \omega \otimes \omega_{ij}^\sharp = a\cdot ([\omega_{ij}\star \omega] \otimes [\omega_{ij}^{\sharp}])\cdot a^{t}$, and it follows from the definition of the operator projective tensor product that
\begin{align*}
\Big\| \sum_{i,j=1}^N \omega_{ij}\star \omega \otimes \omega_{ij}^\sharp\Big\|_{L_1(\F O_F) \widehat{\otimes} L_1(\F O_F)} &\le \|a\|\|a^t\|\|[\omega_{ij}\star \omega]\|_{M_N(L_1(\F O_F))}\|[\omega_{ij}^\sharp]\|_{M_N(L_1(\F O_F))} \\
& \le N\|\omega\|_1\|[\omega_{ij}]\|_{M_N(L_1(\F O_F))}^2,
\end{align*}
where in the second inequality we have used $\|a\| = \|a^t\|  = \sqrt{N}$, the complete contractivity of the convolution product $\star$, and the fact that the involution $\sharp$ is a  complete isometry.  

To conclude the proof, recall that we have a completely isometric identification $M_N(L_1(\F O_F)) \cong \mc{CB}^\sigma(L_\infty(\F O_F), M_N(\C))$, where the latter operator space is the collection of all completely bounded $\sigma$-weakly continuous linear maps from $L_\infty(\F O_F)$ to $M_N(\C)$.  This identification is given by  \[[f_{ij}] \in M_N(L_1(\F O_F)) \mapsto \{x \in L_\infty(\F O_F) \mapsto [\langle f_{ij},x \rangle] \in M_N(\C) \}. \] As a consequence, our matrix $\Omega=[\omega_{ij}] \in M_N(L_1(\F O_F))$ corresponds (under this identification) to the completely positive projection map $L_\infty(\F O_F) = \prod_{k \in \N_0} \mc B(H_k) \to \mc B(H_1)  \cong M_N(\C)$.  In particular, $\|[\omega_{ij}]\|_{M_N(L_1(\F O_F))}  = 1$ and therefore \[\Big\| \sum_{i,j=1}^N \omega_{ij}\star \omega \otimes \omega_{ij}^\sharp\Big\|_{L_1(\F O_F) \widehat{\otimes} L_1(\F O_F)} \le N \|\omega\|_1,\] and a similar argument yields $\Big\|\sum_{i,j=1}^N \omega_{ij}^\sharp \otimes \omega\star \omega_{ij}\Big\|_{L_1(\F O_F) \widehat{\otimes} L_1(\F O_F)} \le N\|\omega\|_1$.
\end{proof}

Since the left and right actions of $L_1(\F O_F)$ extend to actions of $L_1(\F O_F)$ on $L_p(\F O_F)$, we can also consider the map $\Phi$ as a bounded linear map on $L_p(\F O_F)$ for each $1 \le p \le \infty$.  In the following proposition, we let $L_{p,0} \subset L_p(\F O_F)$ denote the codimension $1$ subspace given by the $\|\cdot\|_p$-completion of $\bigoplus_{n \ge 1}\mc B(H_n) = C_c(\F O_F) \ominus \C p_0$.  Note that since $\Phi(L_{p,0}) \subseteq L_{p,0}$ for each $p$, we may regard $\Phi$ as a bounded linear map on $L_{p,0}$.  The following norm estimate is crucial for our proof of unique trace for $C^*_p(\F O_F)$. 

\begin{prop} \label{prop:contractive}
For each $1 < q \le 2$, there is a constant $0 <C(q,N) < 1$ such that 
\[\|\Phi\|_{L_{q,0} \to L_{q,0}} \le C(q,N).\] 
\end{prop}

\begin{proof}
From Lemma \ref{lem:L_1contractive} we have $\|\Phi\|_{L_{1,0} \to L_{1,0}} \le 1$, and in the proof of Theorem 7.2 in \cite{VaVe} it is shown that there exists a constant $0 <C(N) < 1$ such that $\|\Phi\|_{L_{2,0} \to L_{2,0}} \le C(N)$.  Since $L_{q,0}$ is the complex interpolation space $(L_{1,0}, L_{2,0})_{\theta}$, where $\theta = 2(1-q^{-1})$, it follows that \[\|\Phi\|_{L_{q,0} \to L_{q,0}} \le \|\Phi\|_{L_{1,0} \to L_{1,0}}^{1 - \theta}\|\Phi\|_{L_{2,0} \to L_{2,0}}^{\theta} \le C(N)^{2(1-q^{-1})} < 1. \]
\end{proof}

We now have all the tools we need to prove Theorem \ref{thm:uniquetrace}.

\begin{proof}[Proof of Theorem \ref{thm:uniquetrace}] 
Fix $2 \le p < \infty$ and let $1 < q \le 2$ be the conjugate exponent to $p$.  Denote by $\hat h$ the tracial Haar state of $O^+_F = \widehat {\F O_F}$, which corresponds to the positive definite function $p_0 = \varphi^{\lambda}_{\Lambda(p_0),\Lambda(p_0)} \in L_\infty(\F O_F)$.  To prove the theorem, note that it suffices to show \begin{align} \label{eqn:powers}\lim_{k \to \infty}\|\Phi^k(\omega) - \langle \omega,p_0 \rangle \omega_{p_0}\|_{C^*_p(\F O_F)} = 0 \qquad (\omega \in L_1(\G)).
\end{align}  Indeed, if \eqref{eqn:powers} is true and $\tau$ is any $C^*_p(\F O_F)$-norm continuous tracial state on $L_1(\F O_F)$, then $\tau(\omega) = \lim_{k \to \infty} \tau(\Phi^k(\omega))  = \tau(\omega_{p_0})\langle \omega,p_0 \rangle= \langle \omega,p_0 \rangle$ for each $\omega \in L_1(\F O_F)$.  

To prove \eqref{eqn:powers}, fix $x \in C_c(\F O_F)$ and consider $\omega_x \in L_1(\F O_F)$.  (By density, it suffices to prove \eqref{eqn:powers} for elements  of the form $\omega_x$).  Using the notation from Corollary \ref{cor:Lqestimate_operatornorm}, let $n(x) = \max \{n:  p_n x \ne 0\}$.   For each $k \in \N$, let $z_k \in C_c(\F O_F)$ be the unique element such that $\Phi^k(\omega_x - \langle \omega_x,p_0 \rangle\omega_{p_0}) = \omega_{z_k}$.  Using  the fusion rules for $O^+_F$ described in Section \ref{section:prelims_orthogonal}, we get the estimate $n(z_k) \le n(x) +2k$.  In particular, it follows from Corollary \ref{cor:Lqestimate_operatornorm} and Proposition \ref{prop:contractive} that
\begin{align*}
\|\Phi^k(\omega_x) - \langle \omega_x,p_0 \rangle\omega_{p_0}\|_{C^*_p(\F O_F)}& = \|\Phi^k(\omega_x - \langle \omega_x,p_0 \rangle\omega_{p_0})\|_{C^*_p(\F O_F)} = \|\omega_{z_k}\|_{C^*_p(\F O_F)} \\
&\le D_N(n(x)+ 2k +1)^{1+ 1/p}\|\omega_{z_k}\|_q \\
& = D_N(n(x)+ 2k +1)^{1+ 1/p} \|\Phi^k(\omega_x -\langle \omega_x,p_0 \rangle\omega_{p_0})\|_q \\
& \le D_N(n(x)+ 2k +1)^{1+ 1/p} \|\Phi\|_{L_{q,0} \to L_{q,0}}^k\|\omega_x - \langle \omega_x,p_0 \rangle\omega_{p_0}\|_q \\
&\le D_N(n(x)+ 2k +1)^{1+ 1/p} C(q,N)^k\|\omega_x - \langle \omega_x,p_0 \rangle\omega_{p_0}\|_q. 
\end{align*}      
Since $0 <C(q,N) <1$, $\lim_{k \to \infty}\|\Phi^k(\omega_x) - \langle \omega_x,p_0 \rangle\omega_{p_0}\|_{C^*_p(\F O_F)} = 0$. 
\end{proof}

\subsection{Unitary free quantum groups} \label{section:unitary}

In this section we will outline how the ideas developed to study the $L_p$-C$^\ast$-algebras $C^*_p(\F O_F)$ can be adapted to the case of (unimodular) unitary free quantum groups.  

Let 
$F \in \text{GL}_N(\C)$ be such that $F \bar F = \pm 1$.  Recall from \cite{VaWa, Ba} that the \textit{unitary free quantum group} $\F U_F$ is the discrete quantum group whose compact dual quantum group  $U^+_F = \widehat{\F U_F}$ is given by the universal C$^\ast$-algebra 
\[C_u(U^+_F) = C^*\big(\{\hat u_{ij}\}_{1 \le i,j \le N} \ | \ \hat U = [\hat u_{ij}] \text{ and }  F \bar{ \hat U} F^{-1}  \text{ are unitary}\big).\]
The coproduct $\hat \Delta_u$ for $U^+_F$ is defined exactly as it is for $O^+_F$.  In particular, $\hat \Delta_u$ is defined so that $\hat U$ becomes an irreducible unitary representation of $U^+_F$.  It is known that $\F U_F$ is non-amenable if and only if $N \ge 2$ \cite{Ba}. As in the orthogonal case, we have by \cite[Th\'eor\`eme 3]{Ba} that $\F U_F$ is unimodular if and only if our $F$ (which satisfies $F \bar F = \pm 1$) is a unitary matrix.  In this case, using  \cite[Remarque on page 145]{Ba}, we can replace $F$ by the identity matrix without changing the underlying discrete quantum group.  Therefore, for the remainder, we  work with $\F U_N:=\F U_{1_N}$ and $U_N^+:= \widehat{\F U_N}$.   

Let us recall the structure of the irreducible unitary representations of $U^+_N$ obtained in \cite{Ba}.  There is a natural identification between $\Irr(U^+_N)$ and the free semigroup  $\F_2^+$ on two generators, which sends the pair of irreducible unitary representations $(\hat U, \overline{\hat U})$ to the pair of generators of $(\alpha,\beta)$ of $\F_2^+$ and the trivial representation to the empty word $e \in \F_2^+$. The conjugation and fusion rules of the irreducible unitary representations are determined recursively by \[\overline{g\hat U} = \overline{\hat U}\overline{g}, \quad \overline{\hat U g} = \overline{g}\overline{\hat U}, \quad g \hat U \otop \hat U h = g\hat U \hat U h, \quad g \hat U \otop \overline{\hat U} h = g \hat U \overline{\hat U}h \oplus (g \otop h) \qquad (g,h \in \F_2^+).\] 
In particular, each inclusion $\gamma \subset g \otop h$ of irreducible  representations is multiplicity free.

Finally, fix a family $(H_g)_{g \in \F_2^+}$ of representative Hilbert spaces on which the irreducible representations of $U^+_N$ act.  Denote by $\ell:\F_2^+ \to \N_0$ the canonical word length function and define, for each $n \in \N_0$, the subspace \[C_c(\F U_N)_n = \bigoplus_{
\ell(g) = n} \mc B(H_g) \subset C_c(\F U_N).\]  Note that $C_c(\F U_N)_n = p_nL_\infty(\F U_N)$ where $p_n$ is the central projection $p_n = \sum_{\ell(g) = n} p_g$.  

Using the above length function $\ell$ and the corresponding homogeneous subspaces $C_c(\F U_N)_n$ of length $n$, we will prove an $L_q$-property of rapid decay for the quantum group $\F U_N$ similar to Proposition \ref{prop:qHaagerup}.  As in the orthogonal case, the following theorem is a key tool in proving that the $L_p$-C$^\ast$-algebras associated to $\F U_N$ are exotic.  To obtain this theorem and the results that follow, we exploit the close structural relationship between $\F O_F$ and $\F U_F$ discovered in \cite{Ba}.  This allows us to use our work in the previous section on $\F O_F$ to deduce corresponding results for $\F U_F$.  

Recall that $\alpha_q$ denotes the natural left action of $L_1(\F U_N)$ on $L_q(\F U_N)$.

\begin{thm} \label{thm:unitaryRD}
Let  $N \ge 3$.  Then there is a constant $D_N > 0$ depending only on $N$ such that for each $1 < q \le 2$ and each  $n \in \N_0$, 
\begin{align} \label{eqn:q}\|\alpha_q(\omega_x)\|_{\mc B(L_q(\F U_N))} \le D_N(n+1)\|x\|_q \qquad (x \in C_c(\F U_N)_n).\end{align}  Moreover, if $2 \le p < \infty$ is the conjugate exponent to $q$, then  
\begin{align}\label{eqn:p}\|\omega_x\|_{C^*_p(\F U_N)} \le D_N(n+1)\|x\|_q.\end{align}
\end{thm} 

\begin{rem}
Presumably the above result is true for $N=2$ as well. However, our proof technique (which relies heavily on the $\F O_F$ case where $N \ge 3$) only allows us to conclude the result when $N \ge 3$. 
\end{rem} 

As in the orthogonal case, Theorem \ref{thm:unitaryRD} relies on the following ``local''  inequality.  Compare with Lemma \ref{lem:block_inequality}.

\begin{lem} \label{lem:blockUnitary}
Let $N \ge 3$ and $D_N > 0$ the constant given by Lemma \ref{lem:block_inequality}.  Then for any $1 \le q \le 2$ and any $\gamma, g,h \in \F_2^+$, we have
\begin{align}\label{eqn:irrepcase}
\|p_\gamma (\alpha_q(\omega_x)(y))\|_{q} \le D_N \|x\|_{q}\|y\|_{q} \qquad (x \in\mc B(H_g), \  y \in \mc B(H_h)).
\end{align}
\end{lem}

\begin{proof}
Fix $\gamma, g,h \in \F_2^+$ and $x \in \mc B(H_g)$, $y \in \mc B(H_h)$ as above.  We assume that $\gamma \subset g \otop h$ (otherwise $p_\gamma (\alpha_q(\omega_x)y) = 0$ and there is nothing to prove).  Since the inclusion $\gamma  \subset g \otop h$ is multiplicity-free, the arguments in the proof of Lemma \ref{lem:block_inequality} apply in the present context to show that it suffices to obtain the following analogue of inequality \eqref{eqn:schatten_ineq}:
\begin{align} \label{eqn:schatten_ineqU}
\|(V_{\gamma}^{g,h})^*(x \otimes y)V_{\gamma}^{g,h}\|_{S_q(H_\gamma)} \le D_N \Big( \frac{d_{\gamma}}{d_gd_h}\Big)^{1-1/q} \|x\|_{S_q(H_g)} \|y\|_{S_q(H_h)},
\end{align} 
where $V_{\gamma}^{g,h}$ is the (unique up to scaling by $\T$) isometric inclusion of representations $H_\gamma \hookrightarrow H_g \otimes H_h$.  To establish \eqref{eqn:schatten_ineqU}, we will examine three cases.  

Case 1:  If $g \otop h$ is irreducible, then $\gamma = g \otop h$ and \eqref{eqn:schatten_ineqU} holds trivially.  

Case 2:  Next we assume $\bar g = \hat U \overline{\hat U} \hat U...$ and $h = \hat U \overline{\hat U} \hat U...$ (or $\bar g = \overline{\hat U} \hat U\overline{\hat U}...$ and $h = \overline{\hat U} \hat U\overline{\hat U}...$ ) are both alternating words in $\{\hat U, \overline{\hat U}\}$.  Then, as shown in \cite[Section 2]{Ve05} (see also \cite[Page 16]{Ve}), we have identifications  $H_g = H_{\ell(g)}$, $H_h = H_{\ell(h)}$, $H_\gamma = H_{\ell(\gamma)}$ and $V_\gamma^{g,h} = V_{\ell(\gamma)}^{\ell(g),\ell(h)}$, where $V_{\ell(\gamma)}^{\ell(g),\ell(h)}, H_{\ell(g)}, H_{\ell(h)}, H_{\ell(\gamma)}$ are the corresponding objects associated to the orthogonal free quantum group $\F O_N:=\F O_{1_N}$.  Using these identifications, \eqref{eqn:schatten_ineqU} follows immediately from \eqref{eqn:schatten_ineq}.  

Case 3: Finally we consider the most general case.  Given $\gamma \subset g \otop h$, the fusion rules for $\text{Irr}(U^+_N)$ imply that there is a unique triple $(g',h',\tau) \in \F_2^+$ such that $g=g'\tau$, $h = \bar \tau h'$ and $\gamma =g'h'$.  Let $g_1$ (respectively $h_1$) be the longest subword of $g'$ (respectively $h'$) such that $g' = g_1 \otop g''$ (respectively $h' = h'' \otop h_1 $).   Next, let $f$ be the longest subword of $\tau$ such $\tau = \tau' \otop f$.  Writing $g_2 = g''\tau'$ and $h_2 = \bar{\tau'}h''$, we obtain decompositions $g= g_1 \otop g_2 \otop f$,   $h = \bar{f} \otop h_2 \otop h_1$ and $\gamma = g_1 \otop g''h'' \otop h_1$.  By construction $g''h'' \subset g_2 \otop h_2$ is an inclusion of the type given by case 2.  To verify \eqref{eqn:schatten_ineqU} in this general case, let $(e_I)_I$ and $(\overline{e_I})_I$ be orthonormal bases for $H_f$ and $H_{\bar f}$, respectively, with the property that $t = d_{f}^{-1/2} \sum_{I} e_I \otimes \overline{e_I}:\C \hookrightarrow H_f \otimes H_{\bar f}$ is an isometric intertwiner (unit invariant vector).  This is always possible since $U^+_N$ is of Kac type.  Then we obtain the following expression for $V_\gamma^{g,h}$ (after making the identifications $H_{g_2} = H_{\ell(g_2)}$, $H_{h_2} = H_{\ell(h_2)}$,  $H_{g''h''} = H_{\ell(g''h'')}$ and $V_{g''h''}^{g_2,h_2} = V_{\ell(g''h'')}^{\ell(g_2),\ell(h_2)}$ as in case 2):
\begin{align} \label{eqn:V}V_\gamma^{g,h} = (\iota_{H_{g_1} \otimes H_{g_2} } \otimes t \otimes \iota_{H_{h_2 } \otimes H_{h_1}})\circ (\iota_{H_{g_1}} \otimes V_{\ell(g''h'')}^{\ell(g_2),\ell(h_2)} \otimes \iota_{H_{h_1}}).
\end{align}
Using this formula we may now proceed along the lines of the proof of Lemma \ref{lem:block_inequality} (using the notation therein).  Identify $H_{\ell(g_2)}$ with the highest weight subspace of $H_{\ell(g'')} \otimes (\C^N)^{\otimes \ell(\tau ')}$ and identify $H_{\ell(h_2)}$ with the highest weight subspace of $ (\C^N)^{\otimes \ell(\tau ')} \otimes H_{\ell(g'')}$.  Let $\{e_{i,j}\}_{i,j}$ , $\{e_{I,J}\}_{I,J}$ and $\{\overline{e}_{I,J}\}_{I,J}$ be the canonical matrix units for $\mc B( (\C^N)^{\otimes \ell(\tau ')})$, $\mc B(H_f)$ and $\mc B(H_{\bar f})$, respectively. 
Then we can uniquely write $x \in \mc B(H_g)$ and $y \in \mc B(H_h)$ as
\[x= \sum_{i,j,I,J} x_{ij}^{IJ} \otimes e_{ij} \otimes e_{IJ}, \qquad y= \sum_{i,j,I,J} \overline{e}_{IJ} \otimes e_{\check{i}, \check{j}} \otimes y_{ij}^{IJ}, \] where $x_{ij}^{IJ} \in \mc B(H_{g_1} \otimes H_{g''})$ and $y_{ij}^{IJ} \in \mc B(H_{h''} \otimes H_{h_1})$.  Using this decomposition and \eqref{eqn:V}, one readily calculates
\begin{align*}
&\|(V_{\gamma}^{g,h})^*(x \otimes y)V_{\gamma}^{g,h}\|_{S_q(H_\gamma)} \\
&= d_f^{-1} \Big\|(\iota_{H_{g_1}} \otimes V_{\ell(g''h'')}^{\ell(g_2),\ell(h_2)} \otimes \iota_{H_{h_1}})^*\Big(\sum_{i,j,I,J} x_{ij}^{IJ} \otimes e_{i,j} \otimes e_{\check{i},\check{j}} \otimes y_{ij}^{IJ}\Big)(\iota_{H_{g_1}} \otimes V_{\ell(g''h'')}^{\ell(g_2),\ell(h_2)} \otimes \iota_{H_{h_1}})\Big\|_{S_q(H_\gamma)}\\
&\le d_f^{-1} D_N\Big(\frac{d_{\ell(g''h'')}}{d_{\ell(g_2)}d_{\ell(h_2)}}\Big)^{1/2}\Big\|\sum_{i,j,I,J} x_{ij}^{IJ} \otimes y_{ij}^{IJ}\Big\|_{S_q(H_{g_1} \otimes H_{g''} \otimes H_{h''} \otimes H_{h_1})} \quad (\text{using \eqref{eqn:dim_ineq}})\\ 
&\le d_f^{-1} D_N\Big(\frac{d_{\ell(g''h'')}}{d_{\ell(g_2)}d_{\ell(h_2)}}\Big)^{1/2} \|x\|_{S_q(H_g)}\|y\|_{S_q(h_h)}  \quad (\text{by Proposition \ref{prop:Schatteninterpolation}}) \\
&=D_N \Big( \frac{d_{\gamma}}{d_gd_h}\Big)^{1/2} \|x\|_{S_q(H_g)} \|y\|_{S_q(H_h)} \le D_N \Big( \frac{d_{\gamma}}{d_gd_h}\Big)^{1-1/q} \|x\|_{S_q(H_g)} \|y\|_{S_q(H_h)}.
\end{align*}
\end{proof}

We are now ready to prove Theorem \ref{thm:unitaryRD}.

\begin{proof}[Proof of Theorem \ref{thm:unitaryRD}]  First note that, as in the proof of Proposition \ref{prop:Lqestimate_operatornorm}, inequality \eqref{eqn:p} follows from \eqref{eqn:q} and Lemma \ref{lem:Kac_result}.  

To prove \eqref{eqn:q}, we first prove the following inequality: For any triple $(l,n,k) \in \N_0^3$,
\begin{align} \label{eqn:block}
\|p_l (\alpha_q(\omega_x)(y))\|_{q} \le D_N \|x\|_{q}\|y\|_{q} \qquad (x \in C_c(\F U_N)_n, \  y \in C_c(\F U_N)_k).
\end{align}
Fix $x \in C_c(\F U_N)_n$, $y \in C_c(\F U_N)_k$ and $l \in \N_0$.  We write $l \subset n \otimes k$ if there exist $\gamma, g, h \in \F_2^+$ such that $\gamma \subset g \otop h$ and $(\ell(\gamma), \ell(g), \ell(h)) = (l,n,k)$.  Note that to prove \eqref{eqn:block}, we can assume
$l \subset n \otimes k$, since $p_l (\alpha_q(\omega_x)(y)) = 0$ otherwise.   Let $\gamma \subset g \otop h$ with $(\ell(\gamma), \ell(g), \ell(h)) = (l,n,k)$ and let $r = \frac{n+k-l}{2}$.  Then there exists a unique triple $(\tau,g',h') \in (\F_2^+)^3$ with $(\ell(\tau), \ell(g'), \ell(h') = (r,n-r,k-r)$, $g=g'\tau$, $h=\bar \tau h'$ and $\gamma = g'h'$.  Moreover, every such triple $(\gamma, g,h)$ arises in this way.  Using this change of variables, we can write
\[p_l \alpha_q(\omega_x)(y) = \sum_{\substack{\ell(\gamma) = l\\
\ell(g) = n \\ \ell(h) = k}} p_{\gamma}\alpha_q(\omega_{p_{g}x})(p_hy) = \sum_{\substack{\ell(g') = n-r \\
\ell(h') = k-r}} \sum_{\ell(\tau) = r}p_{g'h'}\alpha_q(\omega_{p_{g'\tau}x})(p_{\bar \tau h'}y). \] Since the projections $p_{g'h'}$ are mutually orthogonal for differing values of $g',h'$, we obtain the following $L_q$-norm estimate. 
\begin{align*}
\|p_l \alpha_q(\omega_x)(y)\|_q^q &= \sum_{\substack{\ell(g') = n-r \\
\ell(h') = k-r}}\Big\| \sum_{\ell(\tau) = r}p_{g'h'}\alpha_q(\omega_{p_{g'\tau}x})
(p_{\bar \tau h'}y) \Big\|_q^q\\
&\le D_N^q\sum_{\substack{\ell(g') = n-r \\
\ell(h') = k-r}}\Big(\sum_{\ell(\tau) = r}\|p_{g'\tau}x\|_q \|p_{\bar \tau h'}y\|_q \Big)^q \quad (\text{by Lemma \ref{lem:blockUnitary}})\\
&\le  D_N^q\sum_{\substack{\ell(g') = n-r  \\
\ell(h') = k-r}}\Big(\sum_{\ell(\tau) = r}\|p_{g'\tau}x\|_q^q\Big) \Big(\sum_{\ell(\tau) = r} \|p_{\bar \tau h'}y\|_q^p \Big)^{q/p} \quad (\text{where $p^{-1} +q^{-1} = 1$}) \\
&\le D_N^q\sum_{\substack{\ell(g') = n-r \\
\ell(h') = k-r}}\Big(\sum_{\ell(\tau) = r}\|p_{g'\tau}x\|_q^q\Big) \Big(\sum_{\ell(\tau) = r} \|p_{\bar \tau h'}y\|_q^q \Big) \quad (\text{since $p \ge q$})\\
&=  D_N^q \|x\|_q^q\|y_q\|^q. 
\end{align*}
This proves \eqref{eqn:block}.

The proof of \eqref{eqn:q} using \eqref{eqn:block} is now almost identical to the proof of Proposition \ref{prop:qHaagerup}.  Indeed, since the fusion rules for $\text{Irr}(U^+_N)$ imply that $l \subset  n \otimes k$ iff $l = n+k - 2r$ for some $0 \le r \le \min\{n,k\}$, it follows (just as in the orthogonal case) that the sets $\{k: l \subset n \otimes k \}$ and $\{l: l \subset n \otimes k \}$ have cardinality at most $n +1$.  The rest of the argument follows  the proof of Proposition \ref{prop:qHaagerup} verbatim:  Take $y  = \sum_{k \in \N_0} y_k \in L_q(\F U_N)$ with $y_k = p_k y \in C_c(\F U_N)_k$.  Then $p_l(\alpha_q(\omega_x)y_k \ne 0$ only if $l \subset n \otimes k$, and therefore 

\begin{align*}
\|\alpha_q(\omega_x)(y)\|_q^q &= \sum_{l \in \N_0} \|p_l\alpha_q(\omega_x)(y)\|_q^q  = \sum_{l \in \N_0} \Big\|\sum_{k: l \subset n \otimes k} p_l \alpha_q(\omega_x)(y_k)\Big\|_q^q \\
&\le D_N^q \|x\|_q^q \sum_{l \in \N_0}\Big(\sum_{k: l \subset n \otimes k} \|y_k\|_q\Big)^q \quad (\text{by \eqref{eqn:block}})\\
& \le D_N^q \|x\|_q^q (n+1)^q\|y\|_q^q.
\end{align*}
\end{proof}

Using Theorem \ref{thm:unitaryRD}, we obtain a characterization of the positive definite functions on $\F U_N$ that extend to states on $C^*_p(\F U_N)$ which parallels the orthogonal
case (Theorem \ref{thm:characterization_p_continuity}).

\begin{thm}\label{thm:characterization_p_continuityUnitary}
Let $N \ge 3$ and $2 \le p < \infty$.  The following conditions are equivalent for a positive definite function $\varphi \in B(\F U_N)$: 
\begin{enumerate}
\item \label{FSU} $\varphi \in B_{L_p(\F U_N)}(\F U_N)$.
\item \label{supU} $\sup_{n \in \N_0} (n+1)^{-1}\|p_n\varphi\|_{p} < \infty$.
\item \label{lengthU} $(1 + \ell)^{-1 - \frac{2}{p}}\varphi \in L_p(\F U_N)$.
\item \label{weaklpU} $\varphi$ is weakly $L_p$.  I.e., $r^\ell \varphi \in L_p(\F U_N)$ for each $0 < r < 1$, where $(r^\ell)_{r \in (0,1)} \subset C_0(\F U_N)$ denotes the semigroup 
\[
r^{\ell} = \prod_{g \in \F_2^+} r^{\ell(g)}1_{\mc B(H_g)}.
\]
\end{enumerate}
\end{thm}

\begin{proof}
The implications $\eqref{FSU} \implies \eqref{supU} \implies \eqref{lengthU} \implies \eqref{weaklpU}$ are proved exactly as in Theorem \ref{thm:characterization_p_continuity}.

\eqref{weaklpU} $\implies$ \eqref{FSU}.  In \cite[Section 5.2]{BrAP}, an approximate unit $(\psi_r)_{0 < \kappa < r < 1} \subset C_0(\F U_N)$ consisting of norm one positive definite functions was constructed with the property that $\psi_r \le Cr^{\ell}$ for some $C > 0$ (depending only on $N$) \cite[Proposition 5.8]{BrAP}. (Warning: In \cite{BrAP} this net
is denoted by $(\hat \Psi_t)_t$ where $t \in (t_0,N)$.  See equation (5.10) in \cite{BrAP}.  Here we are taking $\psi_r = \hat\Psi_{rN}$ where $\kappa = t_0/N < r < 1$).   Using the approximate unit $(\psi_r)_r$, we conclude that $\varphi \in B_{L_p(\F U_N)}(\F U_N)$ exactly as in the orthogonal case.
\end{proof}

We can now prove that the C$^\ast$-algebras $(C^*_p(\F U_N))_{p \in (2,\infty)}$ are all exotic.

\begin{cor} \label{cor:exoticU}
Let $N \ge 3$.
\begin{enumerate}
\item \label{surjectiveU} For each $2 \le p < p' <\infty$, the canonical quotient map \[C^*_{p'}(\F U_N) \to C^*_p(\F U_N)\] extending the identity map on $L_1(\F U_N)$ is not injective.  
\item \label{non_isoU} For each $2 <p < \infty$, the C$^\ast$-algebra $C^*_p(\F U_N)$ is not isomorphic (as a C$^\ast$-algebra) to either the universal C$^\ast$-algebra $C_u(U_N^+)$ or the reduced C$^\ast$-algebra $C(U_N^+)$. 
\end{enumerate}
\end{cor}

\begin{proof}
\eqref{surjectiveU}.  We recall from \cite{DaKaSkSo} the notion of a (closed) quantum subgroup $\mathbb H$ a discrete quantum group $\G$.  In particular, $\mathbb H$ is discrete, $\Irr(\hat {\mathbb H}) \subseteq \Irr(\hat \G)$ and $L_\infty(\mathbb H)$ can be identified with the subalgebra $p_{\hat{\mathbb H}}L_\infty(\G)$, where $p_{\hat{\mathbb H}} = \sum_{\alpha \in \Irr(\hat {\mathbb H})}p_\alpha$.  Moreover it follows from \cite[Proposition 2.2]{Ve04} that $p_{\hat{\mathbb H}}$ is implements a unital completely positive multiplier of $L_1(\hat \G)$ in the sense of Remark \ref{rem:pd}.  In particular, $p_{\hat{\mathbb H}}$ is a norm one positive definite function on $\G$ by \cite[Theorem 6.4]{Da11}.  

Now let $\mathbb H$ be the quantum subgroup of $\F U_N$ given by 
\[\Irr(\hat {\mathbb H}) = \{g_{2n}\}_{n \in \N_0} \subset \F_2^+  \quad \text{where} \quad g_{2n} = (\hat U \overline{\hat U})^n.\]  Let $(\psi_r)_{0 < \kappa \le r < 1} \subset C_0(\F U_N)$ be the net of norm one positive definite functions constructed \cite[Proposition 5.8]{BrAP}. Let $\tilde{\psi}_r = \psi_r p_{\hat{\mathbb H}}$.  Then $\tilde{\psi}_r \in B(\F U_N)$ is a norm one positive definite function, and the description of $\Irr(U_N^+)$ given in \cite[Proposition 4.3]{VeVo} gives the following expression for $\tilde{\psi}_r$:
\[p_g \tilde\psi_r = \Bigg\{ \begin{matrix} \frac{S_{2n}(rN)}{S_{2n}(N)}1_{\mc B(H_{g_{2n}})} & \text{if } g = g_{2n} \in \Irr(\hat {\mathbb H})  \\
0 & \text{otherwise}\end{matrix} \qquad (g \in \F_2^+). \]
    
Now fix $2 \le p < p' <\infty$ and consider the positive definite function \[\varphi_{r_0}  = \prod_{n \in \N_0} \frac{S_{2n}(r_0N)}{S_{2n}(N)}1_{\mc B(H_n)}\in B_{L_{p'}(\F O_N)}(\F O_N)\backslash B_{L_{p}(\F O_N)}(\F O_N)\]  constructed in the proof of Theorem \ref{thm:free_orthogonal} \eqref{surjectiveU}.  Let $\tilde \psi_{r_0}$ be the corresponding positive definite function on $\F U_N$.    Since we have $\dim(H_{g_{2n}}) = \dim (H_{2n})$, where $H_{2n}$ is the irreducible representation of $\F O_N$ with label $2n$, it follows from an easy modification of Lemma \ref{lem:ratiotest} (considering only even summands) that $\tilde \psi_{r_0}$ is weakly $L_q$ (for some $1 \le q < \infty$) if and only if $\varphi_{r_0}$ is weakly $L_q$.  Since $\varphi_{r_0}$ is weakly $L_{p'}$ but not weakly $L_{p}$, we conclude from Theorem \ref{thm:characterization_p_continuityUnitary} that $\tilde \psi_{r_0} \in B_{L_{p'}(\F U_N)}(\F U_N)\backslash B_{L_{p}(\F U_N)}(\F U_N)$.  I.e., the canonical quotient $C^*_{p'}(\F U_N) \to C^*_{p}(\F U_N)$ is not injective.      

\eqref{non_isoU}.  The argument here is the same as the orthogonal case. Fix $2 < p < \infty$.   The non-isomorphism $C^*_{p}(\F U_N) \ncong C(U^+_N)$ follows from the simplicity of $C(U^+_N)$ \cite[Theorem 3]{Ba} and the non-simplicity of $C^*_{p}(\F U_N)$ by \eqref{surjectiveU}.  The non-isomorphism $C^*_{p}(\F U_N) \ncong C_u(U^+_N)$ follows from the fact that $C^*_p(\F U_N)$ has no characters.  Indeed, any character on $C^*_p(\F U_N)$ corresponds to a weakly $L_p$ positive definite function $\varphi \in B(\F U_N) \cap \mc U (L_\infty(\F U_N))$, and such a $\varphi$ must satisfy
  \[\|r^\ell \varphi\|_{p}^p  = \sum_{g \in \F_2^+} r^{\ell(g)p} \dim(H_g)^2 < \infty \qquad (0 < r < 1). \]
Since the above sum diverges for $r$ close to $1$, no such $\varphi$ can exist.
\end{proof}

\begin{rem}
At this time we are unable to prove the uniqueness of trace for the C$^\ast$-algebras $(C^*_p(\F U_N))_{p \in (2,\infty)}$.  We suspect that a conjugation by generators approach similar to the orthogonal case should work here as well. 
\end{rem}

\end{document}